\documentclass[10pt]{amsart}
\usepackage{tikz-cd}
\usetikzlibrary{decorations.pathmorphing}
\usepackage[utf8]{inputenc}
\usepackage{amssymb,mathrsfs,mathtools}
\usepackage{color}
\usepackage{graphicx}
\usepackage{caption}
\usepackage{subcaption}
\usepackage{float}
\usepackage[bottom]{footmisc}
\usepackage[colorlinks,hyperindex,linkcolor=blue]{hyperref}
\usepackage[capitalise]{cleveref}
\usepackage{enumitem}
\usepackage{bm} 
\usepackage{varwidth} 
\usepackage{thmtools}
\usepackage{thm-restate}
\usepackage[doi=false,isbn=false,url=false,maxbibnames=99]{biblatex}
\hypersetup{
	pdfauthor = {Anna Fokma, Alvaro del Pino Gomez, Lauran Toussaint},
	pdftitle= {Jiggling: an h-principle without homotopical assumptions}}

\usepackage{comment,xcolor}
\usepackage[normalem]{ulem}

\theoremstyle{plain}
\newtheorem{lemma}{Lemma}[section]
\newtheorem*{lemma*}{Lemma}
\newtheorem{proposition}[lemma]{Proposition}
\newtheorem{theorem}[lemma]{Theorem}
\newtheorem*{theorem*}{Theorem}
\newtheorem{corollary}[lemma]{Corollary}
\newtheorem*{corollary*}{Corollary}

\theoremstyle{definition}
\newtheorem{definition}[lemma]{Definition}
\newtheorem{example}[lemma]{Example}

\theoremstyle{remark}
\newtheorem{remark}[lemma]{Remark}
\newtheorem{notation}[lemma]{Notation}

\AtBeginEnvironment{example}{%
  \pushQED{\qed}%
}
\AtEndEnvironment{example}{\popQED\endexample}
\AtBeginEnvironment{remark}{%
  \pushQED{\qed}%
}
\AtEndEnvironment{remark}{\popQED\endexample}
\AtBeginEnvironment{question}{%
  \pushQED{\qed}%
}
\AtEndEnvironment{question}{\popQED\endexample}

\crefname{enumi}{item}{items} 
\crefname{figure}{Figure}{Figures} 

\newcommand{\pfstep}[1]{\textit{#1}.}

\newcommand{\SD}{{\mathcal{D}}}

\newcommand{\SC}{{\mathcal{C}}}
\newcommand{\SR}{{\mathcal{R}}}

\newcommand{\R}{{\mathbb{R}}}

\newcommand{\N}{\mathbb{N}}

\let\epsilon\varepsilon

\newcommand{\catname}[1]{{\normalfont\textbf{#1}}}

\newcommand{\Top}{\catname{Top}}
\newcommand{\Set}{\catname{Set}}
\newcommand{\sSet}{\catname{sSet}}

\newcommand{\id}{{\operatorname{id}}}

\newcommand{\Image}{{\operatorname{Im}}}

\newcommand{\Hom}{\operatorname{Hom}}
\newcommand{\MapsPS}{C_{\PS}}
\newcommand{\MapsPL}{C_{\PL}}
\newcommand{\Diff}{{\operatorname{Diff}}}

\newcommand{\GammaPS}{\Gamma_{\PS}}

\newcommand{\Op}{{\mathcal{O}p}}

\newcommand{\Sol}[1][\SR]{{\operatorname{Sol}}(#1)} 
\newcommand{\SolF}[1][\SR]{{\operatorname{FSol}}(#1)} 
\newcommand{\SolPS}[1][\SR]{{\operatorname{Sol_{\PS}}}(#1)} 
\newcommand{\sSolPS}[1][\SR]{{\operatorname{sSol_{\PS}}}(#1)} 

\newcommand{\pif}{\pi_f}
\newcommand{\dist}[1]{d_{C^{#1}}}

\newcommand{\SRcont}{\SR_{\operatorname{cont}}}
\newcommand{\SRimm}{\SR_{\operatorname{imm}}}
\newcommand{\SRsubm}{\SR_{\operatorname{subm}}}
\newcommand{\SRtransv}{\SR_{\operatorname{transv}}}

\newcommand{\SRverygenpos}{\SR_{\operatorname{VeryGenPos}}}

\newcommand{\str}{\operatorname{star}}
\newcommand{\cl}{\operatorname{Cl}}
\newcommand{\strN}[1]{\operatorname{star}^{#1}}
\newcommand{\strT}[2][T]{\str(#2,#1)}

\newcommand{\ring}{\operatorname{ring}}
\newcommand{\ringT}[2][T]{\ring(#2,#1)}

\newcommand{\SDelta}{\bm{\Delta}}

\newcommand{\op}{\mathrm{op}}
\newcommand{\Sing}{\operatorname{Sing}}
\newcommand{\face}[1]{\delta_{#1}}
\newcommand{\degen}[1]{\sigma_{#1}}

\newcommand{\lin}{\mathrm{lin}}
\newcommand{\topd}{\mathrm{top}}
\newcommand{\PS}{\mathrm{PS}}
\newcommand{\PL}{\mathrm{PL}}
\newcommand{\maxa}{{\mathrm{max}}}
\newcommand{\mina}{{\mathrm{min}}}

\DeclareMathOperator{\rmax}{{r_\maxa}}
\DeclareMathOperator{\rmin}{{r_\mina}}

\newcommand{\lext}{\operatorname{LinExt}}
\newcommand{\lextD}[1][\Delta]{\operatorname{LinExt}_{#1}}
\newcommand{\join}[2]{\operatorname{join}(#1,#2)}
\newcommand{\joins}[4]{\operatorname{join}_{#1}(#2,#3,#4)}
\newcommand{\interpolate}[4]{\operatorname{interpolate}_{#1}(#2,#3,#4)}

\newcommand{\folconst}[1]{\fol(#1)}
\newcommand{\fol}{\mathscr{F}}

\newcommand{\trans}{{\mathrel{\pitchfork}}} 

\newcommand{\trivbr}[1][n]{\underline{\R^{#1}}}

\newcommand{\sSolPST}[1][\SR]{{\operatorname{sSol_{\PS}^T}}(#1)}

\newcommand{\aspan}{\operatorname{ASpan}}
\newcommand{\lspan}{\operatorname{LSpan}}

\newcommand{\maxcoeff}{{\Lambda}}

\begin{document}

	\title{Jiggling: an h-principle without homotopical assumptions}
	
	\date{\today}
	
	
	\author{Anna Fokma}
	\address{Utrecht University, Department of Mathematics, Budapestlaan 6, 3584~CD Utrecht, The Netherlands}
	\email{annaf.math@outlook.com}
	
	\author{\'Alvaro del Pino}
	\address{Utrecht University, Department of Mathematics, Budapestlaan 6, 3584~CD Utrecht, The Netherlands}
	\email{a.delpinogomez@uu.nl}
	
	\author{Lauran Toussaint}
	\address{VU Amsterdam, Department of Mathematics, De Boelelaan 1111, 1081~HV Amsterdam, The~Netherlands}
	\email{l.e.toussaint@vu.nl}
	
\begin{abstract}
The jiggling lemma of Thurston shows that any triangulation can be jiggled (read: subdivided and then perturbed) to be in general position with respect to a distribution.

Our main result is a generalization of Thurston's lemma. It states that piecewise smooth solutions of a given open and fiberwise dense differential relation $\SR \subset J^1(E)$ of first order can be constructed by jiggling arbitrary sections of $E$. Our statement also holds in parametric and relative form. We understand this as an h-principle without homotopical assumptions for piecewise smooth solutions of $\SR$. 
\end{abstract}

\maketitle
	\tableofcontents
	
	\section{Introduction}

In~\cite{Th1,Th2}, Thurston established the h-principle for foliations. One of the ideas that goes into the proof is to decompose the manifold using a triangulation whose simplices are in general position with respect to a given starting distribution; see \Cref{def:genposTriang}. To construct such triangulations, Thurston states what is now known as the `jiggling lemma':
\begin{lemma*}[Thurston's jiggling lemma \cite{Th2}]
Consider a manifold $M$ endowed with a distribution $\xi$. Any smooth triangulation of $M$ can be subdivided and subsequently perturbed to be in general position with respect to $\xi$, over any given compact subset. Moreover, the perturbation can be assumed to be $C^1$-small. 
\end{lemma*} 
This result has found numerous applications in the construction of geometric structures on manifolds, beyond foliations. These include Haefliger structures with transverse geometry~\cite{LM}, contact structures~\cite{colin1999stabilite,Vo16}, and Engel structures~\cite{CPPP,PV}.

\subsection{Jiggling sections}

The main contribution of the present paper is the following observation: general h-principles in the piecewise smooth category can be established using (a suitable generalization of) jiggling. Concretely, we prove a generalization of Thurston's lemma in which, instead of jiggling a triangulation so that it is in general position, we jiggle a section of a fiber bundle $E \to M$ to yield a piecewise smooth solution of a given differential relation $\SR \subset J^1(E)$.

We assume a metric has been fixed on $J^1(E)$ in order to talk about $C^1$-distances. This allows us to quantitatively control of the size of the jiggling:
\begin{definition}\label{def:jiggling}
Consider a fiber bundle $E \to M$ and let $s,s': M \to E$ be piecewise smooth sections with respect to triangulations $T$ and $T'$ of $M$, respectively. We say that $(s',T')$ is an \textbf{$\epsilon$-jiggling} of $(s,T)$ if $T'$ is a subdivision of $T$ and $d_{C^1}(s,s') < \varepsilon$, where $\varepsilon$ is a positive constant if $M$ is compact, and a positive function otherwise.
\end{definition}

In order for the jiggling process to produce solutions of $\SR \subset J^1(E)$, we need to impose two conditions on the differential relation. First of all, as is common for h-principles, we want the relation $\SR$ to be open as a subset of $J^1(E)$. Secondly, we want it to be fiberwise dense; this means that any jet $\sigma \in J^1(E)$ can be made to lie in $\SR$ using an arbitrarily small perturbation of its slope. We formalize this as follows, where $\pif: J^1(E) \to E$ is the front projection:
\begin{definition} \label{def:fiberweiseDense}
A relation $\SR \subset J^1 (E)$ is \textbf{fiberwise dense} if, for all $x \in E$, the restriction of $\SR$ to the fiber of $\pif$ over $x$ is dense. 
\end{definition}

We now state our main theorem:
\begin{restatable*}{thm}{jigglingMfd} \label{th:jigglingMfd}
Let $M$ be a manifold, possibly open, endowed with a smooth triangulation $T:|K| \to M$. Consider a fiber bundle $E \to M$ and an open and fiberwise dense differential relation $\SR \subset J^1(E)$. Then, given 
\begin{itemize}
    \item a continuous function $\epsilon:M \to \R_+$,
    \item a piecewise smooth section $s: M \to E$, and
    \item a subcomplex $Q \subset M$ of $T$ such that $s|_{\Op (Q)}$ is a solution of $\SR$,
\end{itemize}
there exists an $\epsilon$-jiggling $(s',T')$ of $(s,T)$ such that
\begin{itemize}
    \item $s'$ is a piecewise smooth solution of $\SR$ with respect to $T'$, and
    \item $s'|_{Q} = s|_{Q}$.
\end{itemize}
\end{restatable*}
The idea of the proof is as follows: The conditions we pose on the relation $\SR$ (being open and fiberwise dense) imply that the desired solutions exist over any sufficiently small simplex. As such, after subdividing the triangulation, we have local solutions over all simplices, and it remains to glue them together to a global solution. This is slightly delicate, so we assign colors to the top-dimensional simplices in $M$ such that a simplex and all its adjacent simplices are of a different color. Then, by induction on the color of a simplex, we jiggle the section over all simplices of the same color. If our jiggling is small enough, the solution we constructed over simplices of previous colors is only modified slightly and remains a solution.

The idea of jiggling `one color at a time' was inspired by the coloring argument used by Donaldson in \cite{Do}, which he used to achieve quantitative transversality in the asymptotically holomorphic setting. Coloring allows us to avoid various size and distance estimates in the Grassmannian appearing in Thurston's original proof. It is in this simplification that our argument deviates strongly from his, and therefore it can be more readily generalized.

\subsection{Jiggling as an h-principle}

\Cref{th:jigglingMfd} naturally implies a weak equivalence between the spaces of piecewise smooth solutions of $\SR$ and continuous sections of $E$, with one subtlety. Since different sections are piecewise with respect to different triangulations, the space of triangulations is lurking in the background. However, this is a (semi-)simplicial set, not a topological space.

This leads us to define the simplicial set $\sSolPS$ of piecewise smooth solutions of $\SR$ by letting $(\sSolPS)_n$ be the set of piecewise smooth sections $s: M \times \Delta^n \to E$ that solve $\SR$. We compare this simplicial set to the singular complex $\Sing(\Gamma^0(E))$ of the space $\Gamma^0(E)$ of continuous sections of $E$.
\begin{restatable*}{thm}{wheSSets} \label{th:wheSSets}
Let $\SR \subset J^1(E)$ be an open and fiberwise dense differential relation. The inclusion $\sSolPS \hookrightarrow \Sing (\Gamma^0(E))$ is weak homotopy equivalence of simplicial sets.
\end{restatable*}
We pose a couple of interesting open questions: Can a version of Theorem \ref{th:wheSSets} be established for higher order jets? Moreover, can the assumption of fiberwise density be replaced by some form of ampleness (as is required in convex integration)?

We remark that Theorem \ref{th:wheSSets} is an h-principle without homotopical assumptions. Concretely: since our solutions $s: M \rightarrow E$ are meant to be piecewise smooth, their first jet extensions $j^1s: M \rightarrow \SR$ can have discontinuities along the triangulation. As such, they do not define a homotopy class of a map $M \rightarrow \SR$. This is why no formal data appears in the statement.

The best known example of an h-principle without homotopical assumptions is the arborealization program in symplectic/contact topology~\cite{alvarez2020positive,alvarez2021arborealizationI,alvarez2021arborealizationII}. One of its goals, which is similar in flavor to Thurston's jiggling, is to show that front projections of legendrians can be simplified so that their singularities are of so-called arboreal type. This can then be used to simplify the singularities of skeleta of Weinstein manifolds.

A key feature of arboreal singularities is that there are only finitely many models in each dimension. This is not true for singularities of triangulation type. This motivates a few more open questions:
\begin{itemize}
    \item Can one establish a version of Theorem \ref{th:wheSSets} for solutions with arboreal-type singularities?
    \item What about solutions whose germs along said arboreal singularities are modelled on some finite/countable list?
    \item Similarly, in the triangulated setting, can one establish a version of Theorem \ref{th:wheSSets} in which finitely many models over each simplex are allowed?
\end{itemize}

\subsection{Applications} \label{sec:introAppl}

We now state some consequences of our main theorems.

\subsubsection{General position}

The condition that a triangulation is in general position with respect to a distribution $\xi$ is, of course, triangulation dependent. Because of this we have to do a bit of work to phrase it using an open and fiberwise dense relation. We do this in \cref{sec:JigglingGenPosTriang}, establishing Thurston's jiggling lemma (\cref{cor:jigglingTriangOpen}). Something that is new in this direction is that we establish the result for manifolds that are not necessarily compact, which was not addressed by Thurston~\cite{Th2}. In fact, from a technical standpoint, an interesting feature of the proof is that we have to develop some theory regarding crystalline subdivision in the non-compact setting, which we do in \cref{sec:crystallineConingOff}.

Another application is the construction of triangulations that are in general position with respect to multiple distributions at the same time; a statement that has appeared already in applications, e.g.~\cite{Vo16,PV}.

\subsubsection{Piecewise maps of full rank}

According to Thurston's jiggling lemma, any map between manifolds can be jiggled to be piecewise transverse to a distribution or to be piecewise of maximal rank (immersive/submersive). These statements can be upgraded to full h-principles without homotopical assumptions, for the corresponding differential relations, using our results:
\begin{restatable*}{cor}{wheMaps}
The following inclusions of simplicial sets are weak homotopy equivalences:
    \begin{align*}
        \sSolPS[\SRtransv] &\hookrightarrow \Sing(C^0(M,N)), \\
        \sSolPS[\SRimm] &\hookrightarrow \Sing(C^0(M,N)) \text{ if } \dim M < \dim N, \text{ and } \\
        \sSolPS[\SRsubm] &\hookrightarrow \Sing(C^0(M,N)) \text{ if } \dim M > \dim N.
    \end{align*}
\end{restatable*}

\subsubsection{Piecewise differentiable geometric structures}

\Cref{th:jigglingMfd} tells us that every non-vanishing 1-form is homotopic to a piecewise smooth contact form. Denoting by $\SRcont$ the relation for contact forms, we obtain the following h-principle:
\begin{restatable*}{cor}{appContactWhe}
	The inclusion $\sSolPS[\SRcont] \hookrightarrow \Sing\left( \Gamma^0 \left( \Lambda^1 \left(T^*M \right)\setminus\{0\} \right) \right)  $ is a weak homotopy equivalence.
\end{restatable*}

This complete flexibility in the piecewise smooth setting should be compared to the rigidity that appears in the study of other geometric structures of low regularity. Examples include $C^0$-contact structures~\cite{MuSp,El97wave,RZ20}, $C^0$-symplectic structures~\cite{Op09,HLS15}, Lipschitz symplectic structures~\cite{Jo22}, and PL symplectic structures~\cite{BD24}.

With that said, one may wonder whether one can ``rigidify'' the study of piecewise differentiable contact structures as follows. In dimension $3$, one could study instead the differential relation $\SR^+$ defining positive contact structures (e.g. compatible with a given orientation of the ambient manifold). This relation falls outside the scope of jiggling, since it is not fiberwise dense. We conjecture that there is rigidity in this setting. Moreover, it is unclear to us what to expect in the higher dimensional case.

\subsection{Outline of the paper}

In \cref{sec:prelim} we discuss the preliminaries of the paper, including jet bundles, triangulations and simplicial sets. 

Sections \ref{sec:crysSubdiv} through \ref{sec:perturbing} contain the ground work for jiggling. In \cref{sec:crysSubdiv} we discuss crystalline subdivisions of triangulations, as well as some other subdivision schemes required to work in the relative setting. In \cref{sec:linearizing} we show how we can (locally) approximate a section by its linearization, reducing the jiggling of general sections to the jiggling of piecewise linear maps. In \cref{sec:perturbing} we show that for open and fiberwise dense relations we can tilt any section to yield a local solution over any sufficiently small simplex.

The proof of the main theorem (\cref{th:jigglingMfd}) is contained in \cref{sec:jiggling}. The h-principle without homotopical assumptions that follows is discussed in \cref{sec:whe}. Applications are discussed in \cref{sec:app}.

\subsection{Acknowledgements}

The authors want to thank Miguel Barata for reading a draft of this paper. This paper forms a part of the first author's PhD thesis. Hence, the authors would like to thank the reading committee, consisting of M\'{e}lanie Bertelson, Kai Cieliebak, Ga\"{e}l Meigniez, Ieke Moerdijk and Thomas Rot, for reviewing the thesis and offering their feedback.

The second author is funded by the NWO (Dutch Research Council) project VI.Vidi.223.118. The third author is funded by the NWO project ``proper Fredholm homotopy theory'' (project number OCENW.M20.195) of the research programme Open Competition ENW M20-3.
	\section{Preliminaries} \label{sec:prelim}

In this section we cover some basics and fix notation. We discuss jet spaces and h-principles in \cref{sec:JetHPrinc}. We continue with triangulations in \cref{sec:triangul}, after which we end with a brief account of simplicial sets in \cref{sec:simplSets}.

\subsection{Jet spaces and h-principles} \label{sec:JetHPrinc}

In this section we briefly recall the standard terminology on (first order) jet spaces and h-principles. For a more elaborate overview we refer to \cite{CiElMi,Gr86}.

Given a fiber bundle $E \to M$, its \textbf{first jet bundle} $J^1(E) \to M$ is the bundle of first order Taylor polynomials of sections of $E$. The projection is defined by sending a Taylor polynomial at $x\in M$ to the point $x$. We observe that such a Taylor polynomial at $x$ can be represented by a section locally defined around $x$. Vice versa, every section $s : M \to E$ induces a section $j^1 s : M \to J^1(E)$ which we call its first jet. Since not every section of $J^1(E)$ comes from a section of $E$, we call sections which do arise in this manner \textbf{holonomic}. When interested in maps instead of sections, we define the first jet bundle $J^1(M,N)\to M$ of maps $M\to N$ in an analogous manner.

The terminology of jet spaces has proven to be a natural language to generalize the notion of a partial differential \emph{equation} (PDE) to a partial differential \emph{relation}.

\begin{definition}
    A (first order) \textbf{partial differential relation} is a subset $\SR \subset J^1(E)$.
\end{definition}

We distinguish two notions of solutions of a relation $\SR$. Firstly, there exists a space $\SolF$ of \textbf{formal solutions} whose elements are sections of $J^1(E)$ with image in $\SR$. We endow this space with the weak $C^0$-topology. Secondly, there exists a space $\Sol \subset \Gamma^0(E)$ of \textbf{solutions} whose elements are those sections of $E$ whose jet is a formal solution. We endow the space $\Sol$ with the weak $C^1$-topology.

There is a natural inclusion,
\[ \Sol \hookrightarrow \SolF,\quad s \mapsto j^1s,\]
which allows us to compare the space of solutions $\Sol$ to the space of formal solutions $\SolF$, which is more algebraic in nature. The theory of h-principles concerns itself with the question whether the above inclusion is a weak homotopy equivalence. If it is, we say that $\SR$ satisfies the \textbf{(complete) h-principle}. Gromov's method of flexible sheaves~\cite{gromov1969stable} shows that the h-principle holds for a large class of relations (open and $\Diff$-invariant) on open manifolds. Partial and negative results have also been proven, such as~\cite{Be,El92,taubes1994seiberg}.

Multiple variations of the complete h-principle exist. We mention two. We follow the convention that $\Op(C)$ denotes an unspecified arbitrarily small open neighborhood of $C$.
\begin{definition}
The h-principle for $\SR$ holds
	\begin{itemize}
        \item \textbf{relatively} if given a section $s \in \SolF$ such that $s|_{\Op(C)}$ is holonomic, there exists a holonomic section $s' \in \Sol$, which is homotopic to $s$ relative $C$, for every closed set $C \subset M$.
		\item \textbf{parametrically} if given a continuous family $S : K \to \SolF$, there exists a continuous family $S': K \to \Sol$ homotopic to $S$. Here $K$ is a compact manifold acting as parameter space.
	\end{itemize}
\end{definition}

\subsection{Simplicial complexes} \label{ssec:simplicialComplexes}

In this section we discuss simplicial complexes and linear polyhedra. It is worth pointing out that many of the arguments in the paper are indeed linear in nature. Using them in the manifold setting will involve working chart by chart and taking into account various estimates to bound the non-linearity of the transition functions between them.

\subsubsection{Simplicial complexes}

We define the standard $m$-dimensional simplex $\Delta^m \subset \R^m$ as 
\begin{equation*}
	\Delta^m = \{(t_1,\dots,t_m) \in \R^m \mid \sum_{i=1}^m t_i \leq 1 \text{ and } t_i \geq 0 \text{ for all } i \}.
\end{equation*}
A \textbf{linear simplex} is then any subset of Euclidean space that is affinely isomorphic to a standard simplex. By a \textbf{face} of a linear simplex we refer to a subsimplex of any dimension.

Simplices can be glued along their codimension-1 faces if they form a so-called simplicial complex:
\begin{definition}
A \textbf{simplicial complex} is a locally finite set $K$ of linear simplices in an Euclidean space such that
	\begin{itemize}
		\item if $\sigma \in K$ then its faces are also in $K$, and
		\item if $\sigma,\sigma'\in K$ then $\sigma \cap \sigma'$ is either empty or a face of both $\sigma$ and $\sigma'$.
	\end{itemize}
\end{definition}
We denote by $K^{(\topd)}$ the set of top-dimensional simplices of a simplicial complex $K$, and can also define subcomplexes in an obvious manner.

A linear polyhedron is the topological space that is spanned by a simplicial complex. We introduce the word ``linear'' to make a distinction with polyhedra living in arbitrary manifolds, to be considered later.
\begin{definition}
	A \textbf{linear polyhedron} $P$ is the union of the linear simplices in a simplicial complex $K$. That is, $P = \cup_{\Delta \in K} \Delta$. In this case, we call $K$ a \textbf{triangulation} of $P$ and write $P = |K|$.
\end{definition}
Due to local finiteness, our simplicial complexes will consist of countably many linear simplices. In particular, the corresponding polyhedra may be be non-compact. Moreover, for our purposes it is enough to consider simplicial complexes in which the dimension of the simplices is bounded above. Furthermore, we will often work with simplicial complexes of \textbf{pure} dimension. That is, every simplex should be contained in a simplex of top dimension.

It proves useful to be able to refer to the maximal simplicial complex contained in a subset. Hence we introduce the following terminology.
\begin{definition} \label{def:SCinSubset}
    Let $K$ be a simplicial complex in $\R^N$ and let $U \subset \R^N$ be a subset. We define $K \cap U$ as the \textbf{maximal simplicial subcomplex of $K$ contained in $U$}. That is, $K \cap U$ consists of those $\Delta \in K$ such that $|\Delta| \subset U$. 
\end{definition}

\subsubsection{Adjacency}

The appropriate notion of a neighborhood of a linear simplex in a simplicial complex is that of a star, which is defined using the concept of adjacency. Related is the notion of a ring.
\begin{definition} \label{def:starRing}
Consider a simplicial complex $K$ including a linear simplex (or more generally, a subcomplex) $Q$.
\begin{itemize}
\item Two linear simplices are \textbf{adjacent} if they share a face.
\item The \textbf{star} $\str{Q}$ of $Q$ is the set of all its adjacent simplices and their faces.
\item The \textbf{closure} $\cl(A)$ of subset $A$ of a simplicial complex is the smallest subcomplex containing the subset $A$.
\item We define the \textbf{ring} around $Q$ as $\ring{Q} = \cl(\str (Q) \setminus Q)$.
\end{itemize}  
\end{definition}
We write $\strT[K]{Q}$ and $\ringT[K]{Q}$ if we want to emphasize the simplicial complex we are working with. However, to simplify notation, we only indicate $K$ once in expressions in which it occurs multiple times. For instance, we write $\ring(\strT[K]{Q})$ instead of $\ringT[K]{\strT[K]{Q}}$. 
We also write $\strN{n}$ for the $n$-fold iteration of $\str$.

\subsubsection{Subdivisions and isomorphisms}

The most restrictive notion of equivalence between simplicial complexes reads:
\begin{definition}
Two simplicial complexes $K$ and $K'$ are \textbf{isomorphic} if there exists a homeomorphism $F: |K| \to |K'|$, inducing a bijection between simplices, such that for each $\Delta \in K$ the restriction $F|_\Delta$ is linear.
\end{definition}

A more ``local'' version of equivalence for two simplicial complexes is the following:
\begin{definition} \label{def:subcomplAgree}
    Two simplicial complexes $K$ and $K'$ \textbf{agree} on a subset $Q \subset \R^N$ if there exist subcomplexes $K_1$ and $K_2$ of respectively $K$ and $K'$ such that
    \begin{itemize}
        \item $K_1$ and $K_2$ are isomorphic, 
        \item $K_1$ is the largest subcomplex of $K$ such that $|K_1|\subset Q$, and 
        \item $K_2$ is the largest subcomplex of $K'$ such that $|K_2|\subset Q$.
    \end{itemize}
\end{definition}

Alternatively, we can subdivide any simplicial complex to obtain a new simplicial complex triangulating the same polyhedron.
\begin{definition}
A \textbf{subdivision of a simplicial complex} $K$ is a simplicial complex $K'$ such that each $\Delta \in K$ satisfies $\Delta= \cup_{i \in I} \Delta'_i$ for a finite collection of $\Delta'_i \in K'$.
\end{definition}
We remark that any two triangulations of a given polyhedron have a common subdivision.

\subsubsection{Piecewise linear maps}

Between polyhedra, the natural class of maps to consider is:
\begin{definition}
Let $P$ be a linear polyhedron. A map $f: P \rightarrow \R^n$ is \textbf{piecewise linear} if:
\begin{itemize}
    \item the map $f$ is continuous, and
    \item the restriction of $f$ to each simplex of $K$ is affine, where $K$ is some triangulation of $P$.
\end{itemize}
\end{definition}
If we want to emphasize the role of $K$, we will write that $(f,K)$ is a piecewise linear map or that $f$ is piecewise linear with respect to $K$. We denote the set of piecewise linear maps by $\MapsPL(P,\R^n)$. We topologize it below.

We can thus speak of piecewise linear maps between linear polyhedra, yielding the category of linear polyhedra. Isomorphic simplicial complexes, as defined above, yield isomorphic polyhedra. Conversely, two polyhedra are isomorphic if they admit triangulations that are isomorphic.

\subsubsection{Piecewise smooth maps} \label{ssec:PSmaps}
We also consider the more general class of maps:
\begin{definition}
Let $P$ be a linear polyhedron and $N$ a manifold. A map $f: P \rightarrow N$ is \textbf{piecewise smooth} if it is continuous and, for some triangulation $K$, the maps $f|_\Delta$ are smooth for all $\Delta \in K$.
\end{definition}

The set of piecewise smooth maps will be denoted by $\MapsPS(P,N)$. We note that $\MapsPS(P,\R^n)$ contains $\MapsPL(P,\R^n)$. The set $\MapsPS(P,N)$ can readily be endowed with the $C^0$-topology (either the weak or the strong version), by interpreting it as a subset of $C^0(P,N)$.

To define the $C^r$-topology in $\MapsPS(P,N)$, it is convenient to assume that $P$ is of pure dimension. We fix an auxiliary triangulation $K$ of $P$ to make sense of $J^r(\Delta,N)$ for each top simplex $\Delta \in P$. On each $J^r(\Delta,N)$ we fix a metric, so that we can make sense of the $C^r$-metric on $\MapsPS(\Delta,N)$. Then we define:
\begin{definition} \label{def:CrTopPolyhedron}
Consider a pair of maps $f_1, f_2 \in \MapsPS(P,N)$, each piecewise smooth with respect to a triangulation $K_i$ of $P$, for $i=1,2$. Consider a triangulation $K'$ subdividing $K_1$, $K_2$ and $K$. We write $\dist{r}(f_1,f_2) < \epsilon$ if for every top-dimensional simplex $\Delta \in K'$ we have $\dist{r}(f_1|_\Delta,f_2|_\Delta) < \epsilon$. Here we let $\epsilon$ be a positive constant if $P$ is compact, and a positive function $P \to \R_+$ if $P$ is not compact. 
\end{definition}
We observe that the above distance on $\MapsPS(P,N)$ depends on our choice of metrics on each $J^r(\Delta,N)$ and therefore it also depends on $K$. It does however not depend on the choice of $K'$. The underlying topologies (known as respectively the weak and strong) do not depend on any of these choices.

Analogous to \cref{def:jiggling}, we can now also define an $\epsilon$-jiggling of a piecewise smooth map defined on a linear polyhedron.
\begin{definition}
    Consider a pair of maps $f_1, f_2 \in \MapsPS(P,N)$, each compatible with a triangulation $K_i$ of $P$, for $i=1,2$. Assume that the simplicial complex $K_2$ subdivides $K_1$, and that $\dist1(f_1,f_2) < \epsilon$. Then the map $(f_2,K_2)$ is an \textbf{$\epsilon$-jiggling} of $(f_1,K_1)$.
\end{definition}

\subsubsection{Jet spaces over linear polyhedra} \label{sec:JetPolyhedra}
We can also set up the language of jet spaces and differential relations from \cref{sec:JetHPrinc} for linear polyhedra of pure dimension. We fix a finite simplicial complex $K$ of pure dimension and consider piecewise smooth maps $|K| \rightarrow \R^n$. Having fixed $K$ we can consider the jet spaces $J^r(\Delta,\R^n)$ for each $\Delta \in K^{(\topd)}$. To keep notation light we will write $J^r(|K|,\R^n)$ for their union.

Similarly, in each $J^1(\Delta,\R^n)$ we can define a differential relation $\SR_\Delta$. We then write $\SR$ for their union. We refer to a relation $\SR$ as fiberwise dense and open, if each $\SR_\Delta$ is. Whenever we subdivide $K$ we can pullback these differential relations to the simplices in the subdivision.

\subsubsection{Ordered simplicial complexes}

Later on (\cref{sec:crysSubdiv}) we will consider the crystalline subdvision method for simplicial complexes; it will be defined first for the standard simplex and then be generalized to arbitrary simplicial complexes. In order to do so, we will need to make use of an ordering of the vertices. As such, we introduce:
\begin{definition}
An \textbf{ordered simplicial complex} is a pair consisting of a simplicial complex and a total order in its set of vertices.
\end{definition}
We observe that, in an ordered simplicial complex, each simplex can be uniquely parametrized by the standard simplex in a manner that is affine and order preserving (where the vertices of the standard simplex are ordered using the standard basis in $\R^m$). Hence to each ordered linear simplex we associate two linear maps. To formalize this, we first introduce some notation.

\begin{notation}\label{not:aspan}
	We denote the \textbf{affine span} of a subset $S \subset \R^n$ by $\aspan(S)$. When considering a $d$-simplex $\Delta$, the plane $\aspan(\Delta)$ is the unique affine $d$-plane containing $\Delta$. Similarly, we denote by $\lspan(S)$ the \textbf{linear span} of a subset $S \subset \R^n$.
\end{notation}

\begin{definition} \label{def:SimplexToLinMap}
	A linear, ordered $m$-simplex $\Delta = \langle v_0,\dots,v_m \rangle$ in $\R^N$ defines the unique linear map $T_\Delta: \R^m \to \R^N$ sending the $i$th basis element $e_i$ to $v_i-v_0$. Translating $\Delta$ by $-v_0$ provides us with the set $\{-v_0\} + \Delta$, and by considering its linear span we obtain the linear map $T_\Delta^{-1} : \lspan(\{-v_0\} + \Delta )\to\R^m$.
\end{definition}

\subsubsection{Colored simplicial complexes}

Our main construction will involve arguing over multiple simplices at once, but to do so we need them not to interact. This is possible (\cref{ssec:perturbing}) as long as they are far apart, which we formalize by requiring that their stars are disjoint. To keep track of this we will assign a ``color'' to each simplex so that simplices in a star never have the same color.

For the following definition we use the notation $[C]$ with $C\in\N$ to denote the ordered set $\{0,\dots,C\}$.
\begin{definition} \label{def:coloring}
Let $K$ be a simplicial complex of pure dimension. A \textbf{$C$-coloring} of $K$ is a surjective map 
\[ c: K^{(\topd)} \to [C]\]
such that $c(\Delta) \neq c(\Delta')$ whenever $\operatorname{star}(\Delta) \cap \operatorname{star}(\Delta') \neq \emptyset$.
\end{definition}
We refer to $C$ as the \textbf{size} of the coloring $c$. We write $K^{(\topd,i)} \subset K^{(\topd)}$ for the top simplices of $K$ of color $i$.

\subsection{Triangulations} \label{sec:triangul}

We now recall the notion of a triangulation of a manifold; i.e. a decomposition of the manifold into smooth simplices. This combinatorial description of the manifold is particularly useful for local arguments and, as we shall see, will turn our arguments into arguments about linear polyhedra. In this section we fix our notation and conventions on the topic. For a more detailed account of triangulations we refer to \cite[Ch. IV.B]{whitney2012geometric} and \cite{lurie2009topics}.

\subsubsection{Polyhedra}

Instead of defining triangulations directly, we recall the more general notion:
\begin{definition}
A \textbf{polyhedron} $T: |K| \rightarrow M$ consists of a simplicial complex $K$ and a family of smooth embeddings $(T_\Delta : \Delta \to M)_{\Delta \in K}$ that glue to a homeomorphism $T: |K| \to T(|K|)$.
\end{definition}
Equivalently, a polyhedron is an embedding of a linear polyhedron $|K|$ into $M$ that is piecewise smooth with respect to $K$. A subset $Q \subset M$ is a \textbf{subpolyhedron} of $T$ if there exists a subcomplex $K''$ of a subdivision $K'$ of $K$ such that $T|_{|K''|}:|K''| \rightarrow Q$ is a homeomorphism.

Given a polyhedron $T: |K| \to M$ we define a \textbf{simplex} $\Delta'$ of $T$ as the image of a linear simplex $\Delta \in K$ under $T$. With a slight abuse of notation we will write $\Delta' \in T$. Keeping $K$ in mind, we still speak of faces, stars and rings.

\begin{definition}
A \textbf{triangulation} of a manifold $M$ is a polyhedron $T: |K| \rightarrow M$ with $T(|K|) = M$.
\end{definition}
It is a result of Whitehead \cite{Whitehead1940} that any smooth manifold can be triangulated, which is unique up to a piecewise linear homeomorphism. The name ``Whitehead triangulation'' is sometimes used to emphasize the compatibility between the triangulation and the smooth structure of $M$.

\subsubsection{Piecewise smooth sections}

Given a bundle over a triangulated manifold, we consider:
\begin{definition}
A section $s: M \to E$ of a fiber bundle $E$ is \textbf{piecewise smooth} if there exists a triangulation $T: |K| \to M$ such that $s \circ T$ is piecewise smooth.
\end{definition}
We denote the space of such sections by $\GammaPS(E)$. As before, if we want to emphasize the role of $T$, we may say that $(s,T)$ is a piecewise smooth section.

To define the $C^r$-topology on $\GammaPS(E)$, we again need to fix a metric on the jet space $J^r(E)$. Compared to the setting of maps over linear polyhedra from \cref{def:CrTopPolyhedron}, we however do not need to fix a triangulation to make sense of jet bundles over $M$. Then we obtain:
\begin{definition}
Given a pair of sections $(s_1,T_1),(s_2,T_2) \in \GammaPS(E)$ we write $\dist{r}(s_1,s_2) < \epsilon$ if there exists a common subdivision $T : |K| \to M$ of $T_1$ and $T_2$ such that for every simplex $\Delta \in T$ we have $\dist{r}(s_1|_\Delta,s_2|_\Delta) < \epsilon$. Here we let $\epsilon$ be a positive constant if $P$ is compact, and a positive function $P \to \R_+$ if $P$ is not compact. 
\end{definition}
We point out that, even though the distances defined above will generally differ from those defined in \cref{def:CrTopPolyhedron}, the induced topologies on $\GammaPS(E)$ do coincide with those in \cref{def:CrTopPolyhedron}.

\subsection{Simplicial sets} \label{sec:simplSets}

In this section we briefly recall the necessary theory of simplicial sets. We emphasize that this is a very short overview which does not give a general account of the theory, but only covers the minimum to reach what we need. In particular, we work our way up to the specific characterization of weak homotopy equivalences of simplicial sets appearing in \cref{def:wheTriangleHomotopy}. This section is mostly based on \cite[Ch. 2]{heuts2022simplicial}, to which we also refer for more details. 

We start by defining the simplex category $\SDelta$ which has objects $[n] = \{0,1,\dots,n\}$. Using the natural ordering $0 \leq 1 \leq \dots \leq n$ we define morphisms in $\SDelta$ as non-decreasing functions. These are generated by the face maps $\delta^i : [n-1] \to [n]$ and the degeneracy maps $\sigma^i : [n+1] \to [n]$, which are defined by either skipping or doubling the element $i$, respectively. This allows us to define:
\begin{definition}
A \textbf{simplicial set} is a functor $X: \SDelta^{\op} \to \Set$.
\end{definition}
Unpacking the definition we see that a simplicial set $X$ is completely determined by the sets $X_n \coloneq X([n])$, the face maps $\delta_i \coloneq X(\delta^i): X_n \to X_{n-1}$ and the degeneracy maps $\sigma_i \coloneq X(\sigma^i): X_n\to X_{n+1}$.

To every topological space we can associate a simplicial set by considering its singular complex. Simplicial sets still allow us to talk about weak homotopy type, but are often easier to work with than topological spaces, due to their combinatorial nature.
\begin{definition}
The functor $\Sing : \Top \to \sSet$ is defined by $\Sing(M)_n = \Hom_\Top(\Delta^n,M)$ for all $ M \in \Top$. We refer to $\Sing(M)$ as the \textbf{singular complex} of $M$.
\end{definition}

The geometric realization $|\cdot | : \sSet \to \Top$ functor is also relevant to our discussion. We do not want to define it precisely, but we do want to point out that it can be seen as the inverse of the singular complex in two ways. Firstly, it holds that there exists a bijection
\[\Hom_\Top(|X|,M) \to \Hom_\sSet(X, \Sing(M))\]
which is natural in $X \in \sSet$ and $M \in \Top$. This can be phrased by saying that the singular complex is the right adjoint to the geometric realization. Secondly, we have the following statement:
\begin{lemma} \label{lem:wheSing}
	Let $M \in \Top$, then there exists a natural map $|\Sing(M)| \to M$ which is a weak homotopy equivalence. We denote this as $M \simeq |\Sing(M)|$. 
\end{lemma}

\subsubsection{Weak homotopy equivalences}

In the category of topological spaces a weak homotopy equivalence can be defined as the vanishing of the relative homotopy groups. We can use a similar definition for (a subcategory of) simplicial sets, but for this we need simplicial analogues of the disc and its boundary. These roles are played by the standard $n$-simplex and its boundary. We use the notation $\Delta[n]$ for the standard $n$-simplex as a simplicial set, and the notation $\Delta^n$ as a topological space. 
\begin{definition}
The \textbf{standard $n$-simplex} $\Delta[n]$ is the simplicial set
\[ \Delta[n]  = \Hom_{\SDelta}( \cdot, [n]). \]
Its \textbf{boundary} $\partial \Delta[n]$ is defined by
\[	(\partial \Delta[n])_m = \{ f \in \Hom_{\SDelta}( [m], [n]) \mid f \text{ is not surjective} \}. \]
\end{definition}
We have two observations to make concerning the above definition. First, note that $\Delta[n]$ and $\Sing(\Delta^n)$ are not the same. One could show for instance that the latter is a Kan complex, which we introduce below, whereas the former is not. The geometric realization of $\Delta[n]$ is however the standard simplex $\Delta^n$, which explains the notation. Secondly, we note that the Yoneda lemma gives us a bijection $\Hom_\sSet(\Delta[n], X) \to X_n$ for all $X \in \sSet$. This provides a useful interpretation of morphisms with domain $\Delta[n]$ and clarifies its role in the theory.

We can now make sense of homotopies in the category of simplicial sets (which strongly resembles the definition for topological spaces). In the following definition the product of two simplicial sets is just their product as functors. That is, $(X \times Y )_n = X_n \times Y_n$ for any given $X,Y \in \sSet$.
\begin{definition}
	Two morphisms $f,g: X \to Y$ of simplicial sets are \textbf{homotopic} if there exists a map $\eta: X \times \Delta[1] \to Y$ such that the diagram 
	\[ \begin{tikzcd}
		X \arrow[r,"\simeq"] \arrow[drr,"f"] & X \times \Delta[0] \arrow[r,"\id \times \delta^1"]  & X \times \Delta[1] \arrow[d,"\eta"] & X \times \Delta[0] \arrow[l,"\id \times \delta^0"']  & \arrow[l,"\simeq"'] \arrow[dll,"g"'] X \\%
		& & Y & & 
	\end{tikzcd}
	\]
	commutes.
\end{definition}

To define weak homotopy equivalences for simplicial sets, we restrict ourselves here to Kan complexes, which are simplicial sets that allow for so-called horn-fillings. As a space, we define the $k$-horn $\Lambda^n_k$ of $\Delta^n$ as the subcomplex $\overline{(\partial \Delta^n \setminus F_k)}$ of the standard $n$-simplex $\Delta$ where $F_k$ is the face of $\Delta^n$ excluding the $k$th vertex. It is the geometric realization of its counterpart in simplicial sets, which we denote by $\Lambda[n,k]$.
\begin{definition}
    The \textbf{$k$-horn of the standard $n$-simplex} $\Lambda[n,k]$ is defined as a simplicial set by
	\[
	    (\Lambda[n,k])_i  = \{ f \in \Hom_\Delta( [i], [n]) \mid \Image f \cup \{k\} \neq [n] \}.
	\]
\end{definition}

An example of a Kan complex to keep in mind is the singular complex of a topological space. 
\begin{definition}
    A simplicial set $X$ is a \textbf{Kan complex} if for every $n \geq 1$ and $k \in [n]$ any morphism $\Lambda[n,k] \to X$ can be extended to a morphism $\Delta[n] \to X$.
\end{definition}

Restricting ourselves to Kan complexes allows us to define weak homotopy equivalences for simplicial sets, in a way that resembles the case for topological spaces. The definition stated below is not the standard one but equivalent to it in the case of Kan complexes~\cite{jardine1987simplical,dugger2004weak}. It generalizes the notion for spaces via the singular complex and the geometric realization functor.
\begin{definition} \label{def:wheTriangleHomotopy}
	A morphism $f: X \to Y$ of Kan complexes is a \textbf{weak homotopy equivalence} if and only if for all $n \in \N$ and all commutative diagrams
	\[ \begin{tikzcd}
		\partial \Delta[n] \arrow{r} \arrow{d} & X \arrow["f"]{d} \\%
		\Delta[n] \arrow{r} \arrow[dashrightarrow, "\sigma"]{ur} & Y,
	\end{tikzcd}
	\]
	where $\partial \Delta[n] \to \Delta[n]$ is the standard inclusion, there exists a lift $\sigma: \Delta[n] \to X$ such that the upper triangle commutes and the lower triangle commutes up to homotopy constant along $\partial \Delta[n] \hookrightarrow \Delta[n]$.
\end{definition}

    \section{Subdivisions} \label{sec:crysSubdiv}

Various methods exist to subdivide a triangulation into smaller simplices, each with their own properties. We focus on a method called crystalline subdivision (\cref{ssec:crysSubdiv}), whose main benefit is that simplices do not get too distorted when subdividing (\cref{sec:propCrystalline}).

We then discuss how one covers a polyhedron with nice subpolyhedra; we need this to ultimately establish (in Section \ref{sec:jigglingEuclRel}) a version of jiggling that is relative in the domain. The ideas needed for the case of compact polyhedra are explained in \cref{sec:coveringPolyhedra}; they are based on the use of crystalline subdivision.

In the non-compact case one needs to introduce new subdivision schemes, localized over subpolyhedra. The scheme required for our main result (\cref{th:jigglingMfd}) can be implemented rather naively. We call it ``barycentric cone off'' and we explain it in \cref{sec:coningOff}. For other applications (e.g. general position) we need a scheme with better quantitative control over the shape of our simplices. We call it ``generalized crystalline subdivision'' and we describe it in \cref{sec:crystallineConingOff}.

\subsection{Crystalline subdivision} \label{ssec:crysSubdiv}

There are various ways of defining crystalline subdivision, but here we follow the definition of Thurston from \cite[p. 227]{Th2}.

Crystalline subdivision as described below is based on the observation that we know how to subdivide a cube into smaller cubes. We illustrate the procedure in \cref{fig:crystsubdivT}.

\begin{figure}[h]
	\includegraphics[width=0.4\textwidth,page=4]{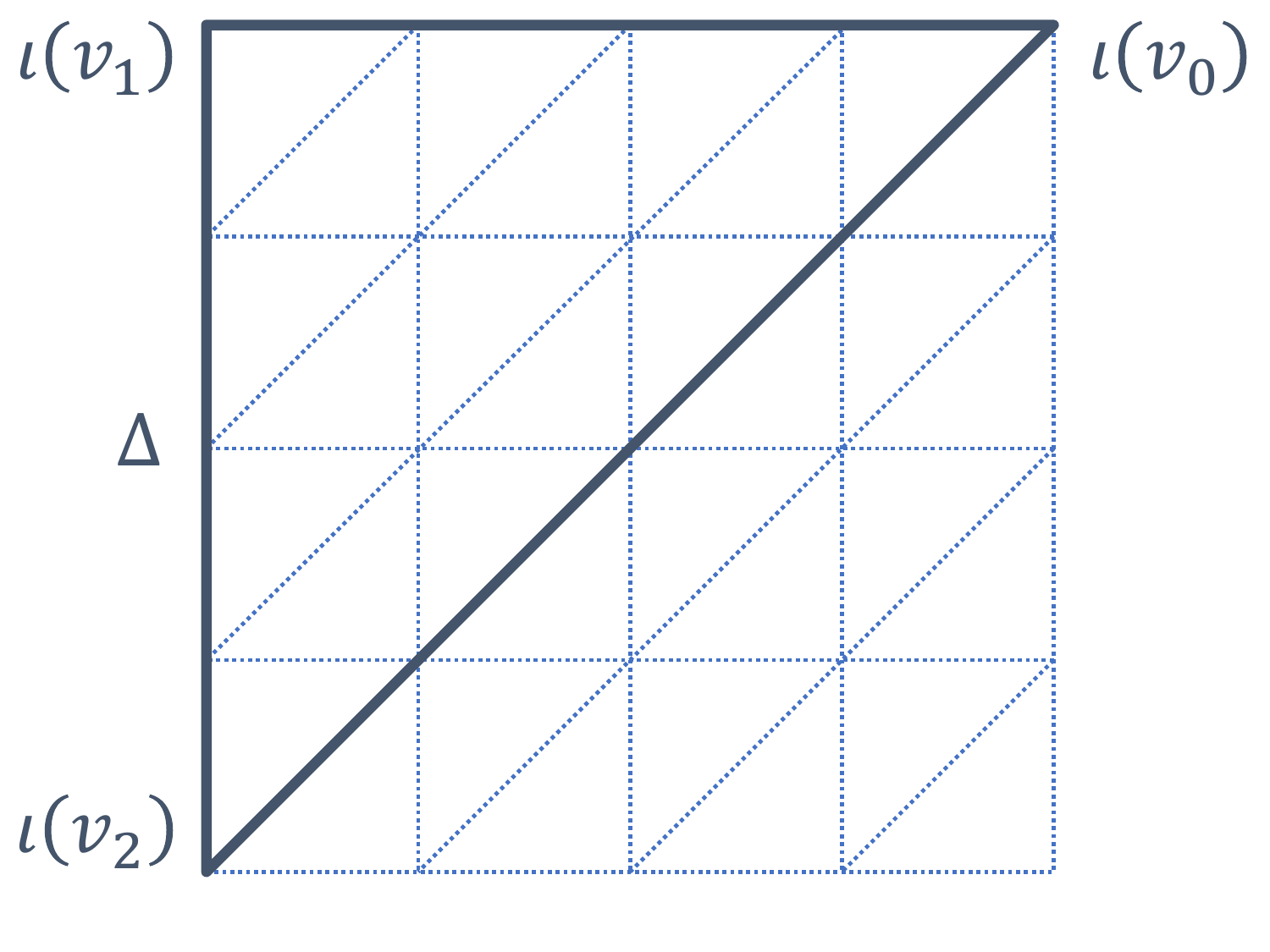}
	\centering
	\caption{The second crystalline subdivision of a $2$-simplex $\Delta$. } \label{fig:crystsubdivT}
\end{figure}

\begin{definition} \label{def:crystsubdivT}
Applying \textbf{crystalline subdivision} to an (ordered) linear $m$-simplex $\Delta$ with $m>0$ is done using the following steps:
	\begin{enumerate}
		\item \label{item:simplexToCube} Include $\Delta =\langle v_{i_0},\dots,v_{i_m} \rangle$ into the standard $m$-cube $I^m$ via the piecewise linear map $\iota$ defined by 
        \[\iota(v_{i_j}) = ({\overbrace{0,\dots,0}^{j }} ,{\overbrace{1,\dots,1}^{m-j }}).\]
		\item \label{item:subdivcube} Subdivide $I^m$ into $2^m$ smaller $m$-cubes of size $1/2$.
		\item Subdivide each of the smaller $m$-cubes into $m!$ smaller linear $m$-simplices. Each linear simplex corresponds to a permutation $\pi$ of $\{0,1,\dots,m\}$ by identifying a permutation $\pi$ with the subset $\{(x_1,\dots, x_m) \in \R^m \mid 0\leq x_{\pi(0)} \leq x_{\pi(1)} \leq \dots \leq x_{\pi(m)} \leq 1\}$. 
        \item Use $\iota$ to pullback the above subdivision of $\iota(\Delta)$ to $\Delta$.
	\end{enumerate}
We define the crystalline subdivision of $\Delta^0$ as itself.
\end{definition}

This procedure makes use of the ordering of the vertices in step~\ref{item:simplexToCube}. Indeed, for $m\geq3$, the resulting subdivision would be different if we applied an affine isomorphism that permutes the vertices of the simplex $\Delta$. It is possible to assign an ordering to the result of crystalline subdivision of $\Delta$, but as we have no use for this, we do not. 

This procedure can be iterated but it turns out to be easier to define finer crystalline subdivisions directly. Hence we define the \textbf{$\ell$th crystalline subdivision} of the standard $m$-simplex $\Delta^m$ by instead subdividing the $m$-cube into $2^{\ell m}$ smaller $m$-cubes of size $2^{-\ell}$ in \cref{item:subdivcube} from \cref{def:crystsubdivT}.

To generalize to a simplicial complex $K$, we choose an ordering on $K$ such that the resulting subdivision is well-defined.
\begin{definition}
Let $K$ be an ordered simplicial complex. Its \textbf{$\ell$th crystalline subdivision} $K_\ell$ is defined as the union of the $\ell$th crystalline subdivisions of its simplices.
\end{definition}
Given a polyhedron $T: |K| \rightarrow M$ associated to a specified ordered simplicial complex $K$, we can also speak of its $\ell$th crystalline subdivision $T_\ell$, seen as the collection of embeddings $(T_\Delta : \Delta \to M)_{\Delta \in K_\ell}$ given by restricting $T$.

\subsection{Properties of crystalline subdivision} \label{sec:propCrystalline}

We discuss some of the properties of crystalline subdivision next. They all have the same flavor: since crystalline does not distort simplices, we are able to obtain various quantitative bounds independent of the order of subdivision $\ell$.

\subsubsection{Coloring}

The first advantage is that we are able to bound the maximum number of colors we need to color a simplicial complex and all its crystalline subdivisions.
\begin{lemma} \label{lem:coloringCrystalline}
Let $K$ be a finite and ordered simplicial complex. Then, there exists a $C\in \N$ such that every crystalline subdivision $K_\ell$ admits a coloring of size $C$.
\end{lemma}
\begin{proof}
Consider the standard $m$-cube $I^m$ subdivided into smaller $m$-cubes of length $2^{-\ell}$ for some $\ell \in \N$. Let $Q \subset I^m$ be one of those smaller cubes. Then $Q$ is adjacent to at most $3^m-1$ cubes. In the procedure of crystalline subdivision each of these cubes is subdivided into $m!$ smaller $m$-simplices. If $A$ is the number of simplices in $K^{(\topd)}$, then each simplex $\Delta \in K^{(\topd)}_\ell$ is adjacent to at most $(3^m-1) \cdot (m)! \cdot  A$ simplices, which means that there exists a $((3^m-1) \cdot (m)! \cdot  A)$-coloring of $K_\ell$.
\end{proof}

\subsubsection{Model simplices} \label{sec:modelSimpl}

A second property of crystalline subdivision is that each of the simplices in the $\ell$th crystalline subdivision of a simplex $\Delta$ is equivalent to a simplex in the first subdivision of $\Delta$, up to scaling and translation. Model simplices appear already in Thurston's work~\cite{Th2}, albeit in a different formulation. 

\begin{lemma} \label{lem:csubdivExact}
Let $K$ be a finite, ordered simplicial complex of pure dimension $m$. Then, there exists a finite collection of \textbf{model simplices} $\SC = \{\Delta_i \subset \R^N \mid i = 0,\dots, I\}$ with the following property: for any $\ell \in \N$ and $\Delta \in K_\ell^{(\topd)}$ there exists 
	\begin{itemize}
		\item a model simplex $\Delta_i \in \SC$, and
		\item a map $t : \R^N \to \R^N$ which is a composition of a translation and a scaling by $2^\ell$
	\end{itemize}
such that $\Delta_i = t(\Delta)$.
\end{lemma}
\begin{proof}
    Let $\Delta \in K^{(\topd)}$ be a top-dimensional simplex, then we define $\SC'_\Delta$ to be the first crystalline subdivision of $\Delta$. We note that any simplex in $(\Delta)_\ell$ is by construction a translated and scaled copy of a unique simplex in $(\Delta)_1$. We define $\SC$ to be the union over all $\SC_\Delta$ for $\Delta \in K^{(\topd)}$.
\end{proof}

\subsubsection{Shape of simplices} \label{sec:shapeSimplices}

In this section we introduce three quantities related to a simplex, which tell us about the size and shape of a simplex. 

To start with, given a simplex, we are interested in the maximal ($\rmax$) and minimal ($\rmin$) distances between a vertex and the face opposite to it. The former quantity controls for instance how well a map is approximated by its linearization with respect to a given triangulation (see \cref{sec:linearizing}). Their ratio on the other hand tells us how degenerate the simplex is. 

For $\rmax$ we note that the maximum distance between a vertex and opposite face agrees with the maximum length of the edges adjacent to both the vertex and the opposite face. For $\rmin$ we interpret the distance between vertex and opposite face as the distance between the vertex and the affine plane spanned by the face. Below we define $\rmax$ and $\rmin$ for $(m+1)$-tuples of points in $\R^N$, but we can also speak of $\rmax$ and $\rmin$ of a linear simplex $\Delta$ by identifying $\Delta$ with its vertices.

\begin{definition} \label{def:rminmax}
	The functions \[\rmin, \rmax : \R^N \times \dots \times \R^N = \R^{N (m+1) } \to \R\] are defined by
	\begin{align*}
		\rmin(p_0,\dots, p_m) &= \min_{i \in [m]} d \left(p_i, \aspan \left( \langle p_0, \dots, \hat{p_i},\dots,p_m \rangle \right) \right) \text{ \quad and} \\
		\rmax(p_0,\dots, p_m) &= \max_{i,j \in [m]} d(p_i,p_j) 
	\end{align*}
\end{definition}

We recall from \cref{def:SimplexToLinMap} that each $m$-simplex $\Delta$ in $\R^N$ defines a unique linear map $T_\Delta: \R^m \to \R^N$, along with its inverse $T_\Delta^{-1}$. Therefore, we can also study the shape of $\Delta$ by studying the operators $T_\Delta$ and $T_\Delta^{-1}$, and in particular their norm. We bound in \cref{lem:rminmaxOperator} the norm of $T_\Delta$ using $\rmax$, but for $\| T_\Delta^{-1} \|$ we introduce an additional quantity $\maxcoeff$. 
As in the case of $\rmax$ and $\rmin$, we define $\maxcoeff$ for $(m+1)$-tuples of points in $\R^N$. The difference is however that here we need to assume that the simplex $\Delta$ is ordered to be able to speak of $\maxcoeff(\Delta)$, since $\maxcoeff$ is only invariant under permutations of its input fixing the first element. 
\begin{definition}\label{def:maxcoeff}
	The function \[\Lambda : \left\{(v_0,\dots,v_m) \in \R^N \times \dots \times \R^N = \R^{N (m+1) } \mid v_i \neq v_j \text{ for all } i \neq j \right\}  \to \R\] is defined by
	\[ \Lambda(v_0,\dots,v_m) = \max_{\substack{  \lambda_1,\dots,\lambda_m \in \R \\ |\sum \lambda_i (v_i-v_0)| =1} }  |\lambda_i|.\]
\end{definition}
We point out that if $v_0=0$ and the points $v_1,\dots,v_m$ form an orthogonal frame, the quantity $\Lambda(v_0,\dots,v_m)$ equals $1/\rmin(v_0,\dots,v_m)$.\footnote{For the interested reader, we note that a closed form expression can be obtained for $\maxcoeff$ by using the Lagrange multiplier method. If $m=2$ and $v_0=0$, we obtain for instance that
\[\Lambda(0,v_1,v_2) = \frac{\max\{|v_1|,|v_2|\}}{\sqrt{|v_1|^2|v_2|^2-\langle v_1,v_2 \rangle^2}}.\] In this case we see that the denominator resembles the Cauchy-Schwarz inequality and hence measures the degeneracy of the simplex $\langle 0, v_1,v_2\rangle$. } 

Having defined the quantity $\maxcoeff$, we obtain the following bounds on the linear maps associated to an ordered, linear simplex.
\begin{lemma} \label{lem:rminmaxOperator}
	Let $\Delta$ be an ordered, linear $m$-simplex in $\R^N$, and let $T_\Delta : \R^m \to \R^N$ be its associated linear map. Then the following bounds hold:
	\begin{itemize}
		\item $ \| T_\Delta \| \leq m \cdot \rmax(\Delta)$, and
		\item $ \| T_\Delta^{-1} \| = m \cdot \maxcoeff(\Delta)$.
	\end{itemize}
\end{lemma}
\begin{proof}
	We write the standard simplex $\Delta^m$ as $\langle e_0,\dots,e_m \rangle$ and the simplex $\Delta$ as $\langle v_0,\dots,v_m \rangle$. We assume without loss of generality that $0=e_0=v_0$. Then the map $T_\Delta$ is defined by $e_i \mapsto v_i$. 
	We observe that
	\begin{align*}
		\|T_\Delta\| 
		= \max_{\substack{\mu_1\dots,\mu_m \in \R, \\ |\sum \mu_i e_i| =1} } \left| T_\Delta \left( \sum_{i=1}^m \mu_i e_i \right) \right| 
		\leq \max_{\substack{\mu_1\dots,\mu_m \in \R, \\ |\sum \mu_i e_i| =1} }  \sum_{i=1}^m |\mu_i| | v_i |  
		\leq \sum_{i=1}^m 1 \cdot \rmax(\Delta)
		= m \cdot \rmax(\Delta) 
	\end{align*}
	and similarly that
	\begin{align*}
		\|T_\Delta^{-1}\| 
		&= \max_{\substack{\mu_1\dots,\mu_m \in \R, \\ |\sum \mu_i v_i| =1} }  \left| T_\Delta^{-1} \left( \sum_{i=1}^m \mu_i v_i \right) \right| 
		\leq \max_{\substack{\mu_1\dots,\mu_m \in \R, \\ |\sum \mu_i v_i| =1} }  \sum_{i=1}^m |\mu_i | | e_i | 
		\leq \sum_{i=1}^m \max_{\substack{\mu_1\dots,\mu_m \in \R, \\ |\sum \mu_i v_i| =1} } |\mu_i | 
		= m \cdot \maxcoeff(\Delta) .\qedhere
	\end{align*}
\end{proof}

A priori we cannot bound the quantities $\rmin$, $\rmax$ and $\maxcoeff$ among all subdivisions of a linear polyhedron, since the space of all linear simplices in $\R^N$ is not compact. However, due to the existence of the model simplices (from \cref{lem:csubdivExact}), we can bound these quantities when considering crystalline subdivisions. We point out that the product of $\rmax$ and $\maxcoeff$ is in particular independent of the number of subdivisions. Additionally, we observe that $\rmin$ and $\rmax$ decrease when we apply crystalline subdivisions, whereas $\Lambda$ increases.
\begin{lemma} \label{lem:rmaxminBound}
	Let $K$ be simplicial complex that is ordered and finite. Then, there exists $B,C,D,E \in \R_+$ such that for all $\ell \in \N$:
	\begin{align*}
		\min_{\Delta \in K_\ell} \rmin(\Delta) = B \cdot 2^{-\ell}&, \quad \max_{\Delta \in K_\ell} \rmax(\Delta) = C \cdot 2^{-\ell}     \\   
		\max_{\Delta \in K_\ell} \Lambda(\Delta) = D \cdot 2^{\ell}   \quad \textrm{ and }& \quad    \max_{\Delta \in K_\ell} \rmax(\Delta) \cdot  \maxcoeff(\Delta) = CD .
	\end{align*}
\end{lemma}

\subsection{Nice covers of polyhedra} \label{sec:coveringPolyhedra}

Our arguments often have to be localized to subpolyhedra. To this end, it is important for us to be able to cover a given polyhedron by subpolyhedra that are nice. We explain how to do this now.

\subsubsection{Nice subcomplexes}

Before we get to the key definition, recall that the join between two subsets $A,B$ of $\R^N$ is the set
\[ \join{A}{B} = \{ t a + (1-t) b \mid a \in A , b \in B \text{ and } t \in [0,1]\}. \]
In particular, when we take the join of a (suitable) pair of linear simplices $\Delta_1,\Delta_2$ we end up with a higher-dimensional simplex $\Delta \coloneq \join{\Delta_1}{\Delta_2}$ having the two original simplices as opposing faces.

\begin{definition} \label{def:niceSubcomplex}
Let $K$ be a simplicial complex and let $K'$ be a subcomplex. We will say that $K'$ is \textbf{nice} if for each simplex $\Delta \in \str(K')$ the subcomplex $\Delta \cap K'$ is a face of $\Delta$.
\end{definition}

\begin{figure}[h]
    \centering
    \begin{subfigure}[t]{0.49\linewidth}
        \centering
        \includegraphics[width=\linewidth,page=9, clip=true, trim = 0 3cm 0 0cm]{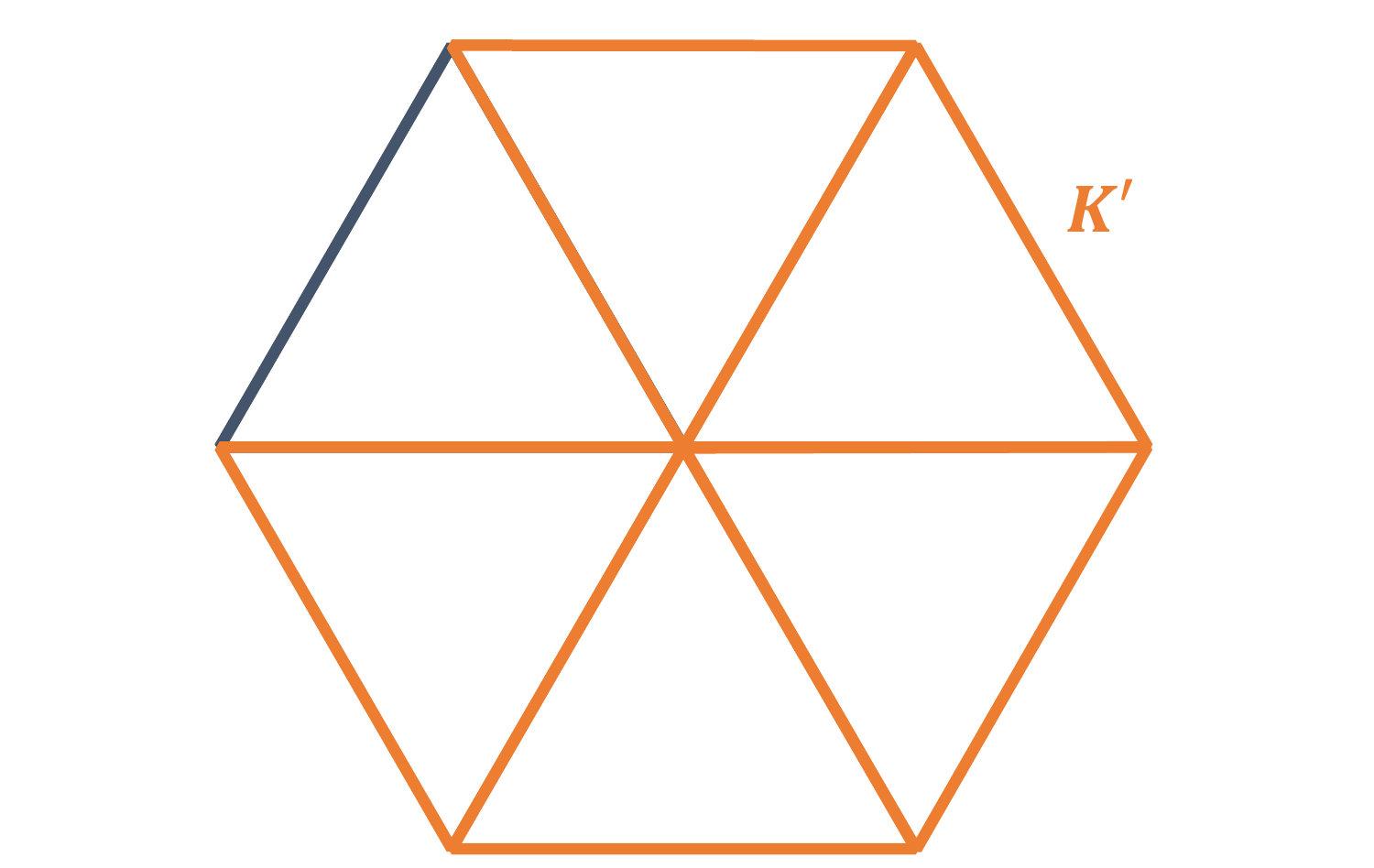}
    \end{subfigure}
    \hfill
    \begin{subfigure}[t]{0.49\linewidth}
        \centering
        \includegraphics[width=\linewidth,page=10, clip=true, trim = 0 3cm 0 0cm]{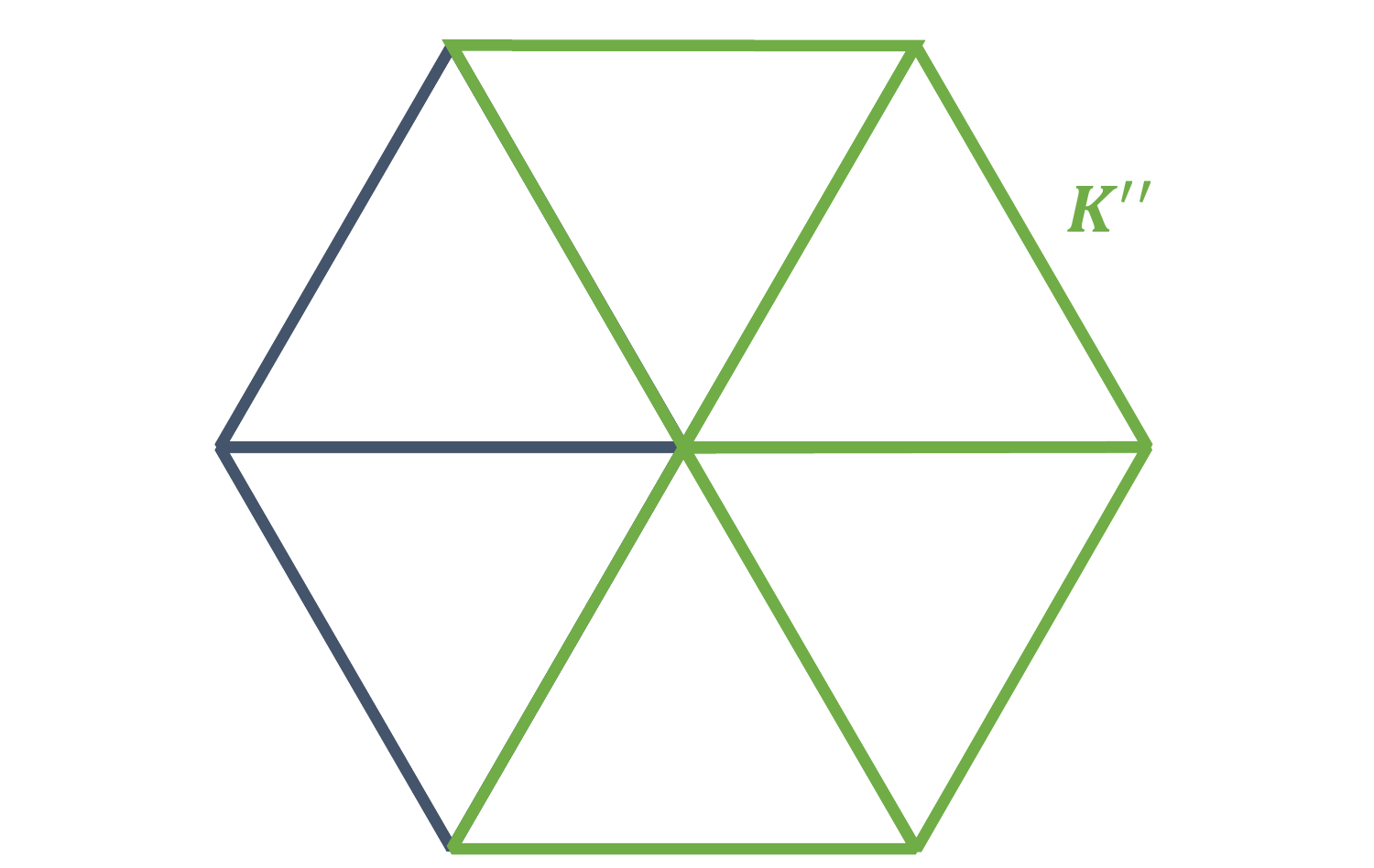}
    \end{subfigure}
	\centering
	\caption{We illustrate here subcomplexes $K'$ and $K''$ of a simplicial complex $K$, where $K'$ is not nice and $K''$ is.} \label{fig:niceSubcomplex}
\end{figure}

We have illustrated \cref{def:niceSubcomplex} in \cref{fig:niceSubcomplex}. The meaning of niceness is that any $\Delta$ in the ring of $K'$ can be thus seen as the join of two faces, $A$ and $B$, with $A \in K'$ and $B$ disjoint from $K'$. This is useful in order to interpolate from a map/section given over $A$ to a map/section given over $B$ (as in \cref{ssec:relativeLinearization}). The following provides a useful criterion for niceness.
\begin{lemma} \label{lem:nicenessCriterion}
The following properties are equivalent for a subcomplex $K' \subset K$:
\begin{itemize}
\item $K'$ is nice.
\item Every $\Delta \in K$ that equals the convex hull of $\Delta \cap K'$ is an element of $K'$.
\end{itemize}
\end{lemma}
\begin{proof}
Suppose $K'$ is nice. Let $\Delta \in K$ equal the convex hull of $\Delta \cap K'$. Then $\Delta \in \str(K')$ and hence by assumption $\Delta \cap K'$ is a face of $\Delta$. Since a face is convex, it follows that $\Delta = \Delta \cap K'$ and hence $\Delta \in K'$.

Conversely, given $\Delta \in \str(K')$ we consider the vertices in $\Delta \cap K'$. They span a face $F$ of $\Delta$, which a priori does not need to be an element of $K'$. For niceness, it suffices to show that $\Delta \cap K'= F$. Note that $F$ is convex and therefore by assumption is an element of $K'$. Hence $F \subset \Delta \cap K'$ and since we cannot add any extra vertex in $\Delta \cap K'$ to $F$, we have the desired equality.
\end{proof}

\subsubsection{Obtaining nice subcomplexes}

The question now is how to cover a simplicial complex by nice subcomplexes. To achieve this we exploit the usual notion of convexity in Euclidean space. We recall from \cref{def:SCinSubset} that $K \cap U$ is the maximal simplicial subcomplex of $K$ contained in a given subset $U$.
\begin{lemma} \label{lem:niceSubcomplex}
Let $K$ be a simplicial complex in $\R^N$ and let $U \subset \R^N$ be a convex subset. Then the intersection $K \cap U$ is a nice subcomplex.
\end{lemma}
\begin{proof}
We check the second condition of \cref{lem:nicenessCriterion}. Let $\Delta \in K$ equal the convex hull of $\Delta \cap (K \cap U)$. Then the vertices of $\Delta$ are in $K \cap U$ and hence also its span $\Delta$. 
\end{proof}

Naively, we could now consider locally finite covers $\{U_i\}$ of $\R^N$ by convex opens and look at the nice subcomplexes $\{K \cap U_i\}$. The issue is that these may not cover $K$ if the top simplices of $K$ are too large compared to the chosen opens. However, there is an easy fix:
\begin{lemma} \label{lem:niceCover}
Let $K$ be a finite simplicial complex in $\R^N$. Let $\{U_i\}$ be a locally finite cover of $\R^N$ by convex opens. Then, $\{K_\ell \cap U_i\}$ is a cover of $K_\ell$ by nice subcomplexes, if $\ell$ is large enough.
\end{lemma}
\begin{proof}
Since $K$ is finite, the set $|K|$ is compact and hence there exists some strictly positive Lebesgue's number $\delta$ associated to the cover $\{ U_i \cap |K|\}$. We now take $\ell$ large enough such that $\rmax(\Delta) < \delta$ for every $\Delta \in K_\ell$.
\end{proof}

\subsection{Barycentric cone off} \label{sec:coningOff}

In light of Lemma \ref{lem:niceCover}, the question remains how to handle simplicial complexes that are not finite. Indeed, without finiteness, we have no bound on the number of subdivisions $\ell$ required to construct a cover by nice subcomplexes.

\begin{figure}[h]
    \centering
    \includegraphics[width=0.5\textwidth,page=5]{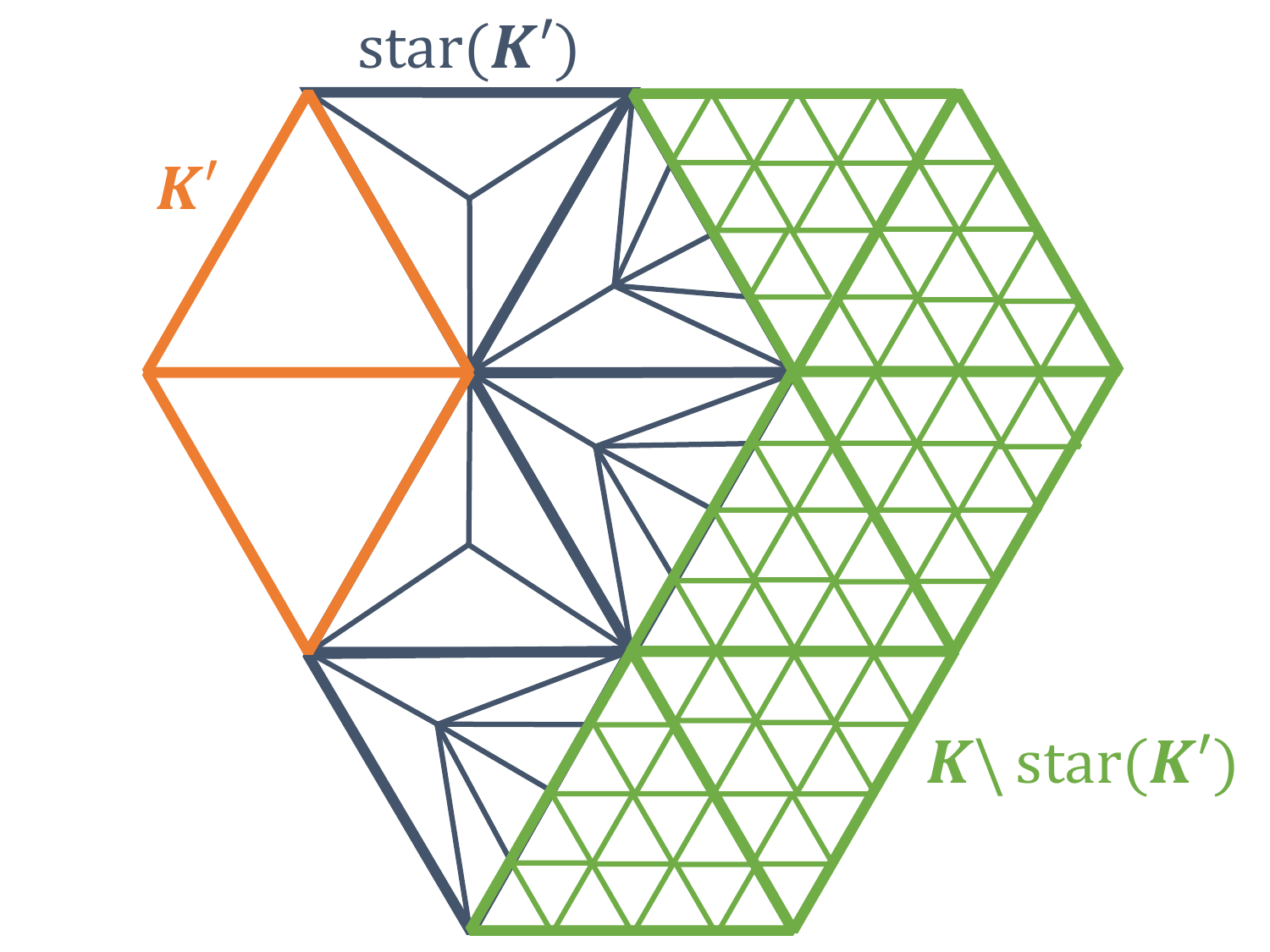}
    \caption{The subcomplex $K \setminus \str(K')$ of $K$ is subdivided in a crystalline manner, but $K'$ not. Hence we apply the barycentric cone off within $\ring(K')$ with respect to an arbitrary crystalline subdivision of $K \setminus \str(K')$.} \label{fig:barycentricConeOff}
\end{figure}

To address this, we must work relatively. Concretely, we need a scheme capable of subdividing the simplicial complex $K$ over a given region, while keeping it the same elsewhere. This will allow us to keep subdividing $K$ as we deal with an infinite cover of $|K|$. It turns out that we can do this rather naively:
\begin{definition} \label{def:barycentricConeOff}
Let $\Delta \subset \R^N$ be a linear simplex. Suppose we are given a subdivision $S$ of its boundary. The \textbf{barycentric cone off} of $\Delta$ with respect to $S$ is obtained by taking the join of the barycenter of $\Delta$ with $S$.
\end{definition}
Hence given a (nice) subcomplex $K'$ of $K$ and a subdivision $K''$ of $K'$, we can apply the barycentric cone off in each simplex in the ring around $K'$. This results in a subdivision of the whole of $K$ that leaves the complement of $\str(K')$ unchanged. This is depicted in Figure \ref{fig:barycentricConeOff}.

\subsection{Generalized crystalline subdivision} \label{sec:crystallineConingOff}

The barycentric coning off procedure described in \cref{def:barycentricConeOff} is enough for us to establish our main jiggling result, also in the case of non-compact manifolds (\cref{sec:jigglingMfd}). However, it turns out not to be well suited to achieving general position (as in Thurston's jiggling lemma) over non-compact manifolds. The reason is that for this application we need some control (in terms of model simplices) as we keep coning off.

Hence we now present a way of coning off where model simplices are preserved, in the sense that we can construct model simplices for a subcomplex $K' \subset K$ before knowing what crystalline subdivision we want outside of $K'$. However, this means we do need a region in which we cone off, which in general will be larger than just a ring. We illustrate the construction in \cref{fig:gencrystsubdivT}.

For the following statement, to ease the notation, we will write $B(A,r)$ for the $r$-neighborhood of a subset $A \subset \R^N$. We also recall that we have defined when simplicial complexes agree in \cref{def:subcomplAgree}.
\begin{figure}[h]
	\includegraphics[width=0.5\textwidth,page=11, clip=true, trim=0 0 0.1cm 0]{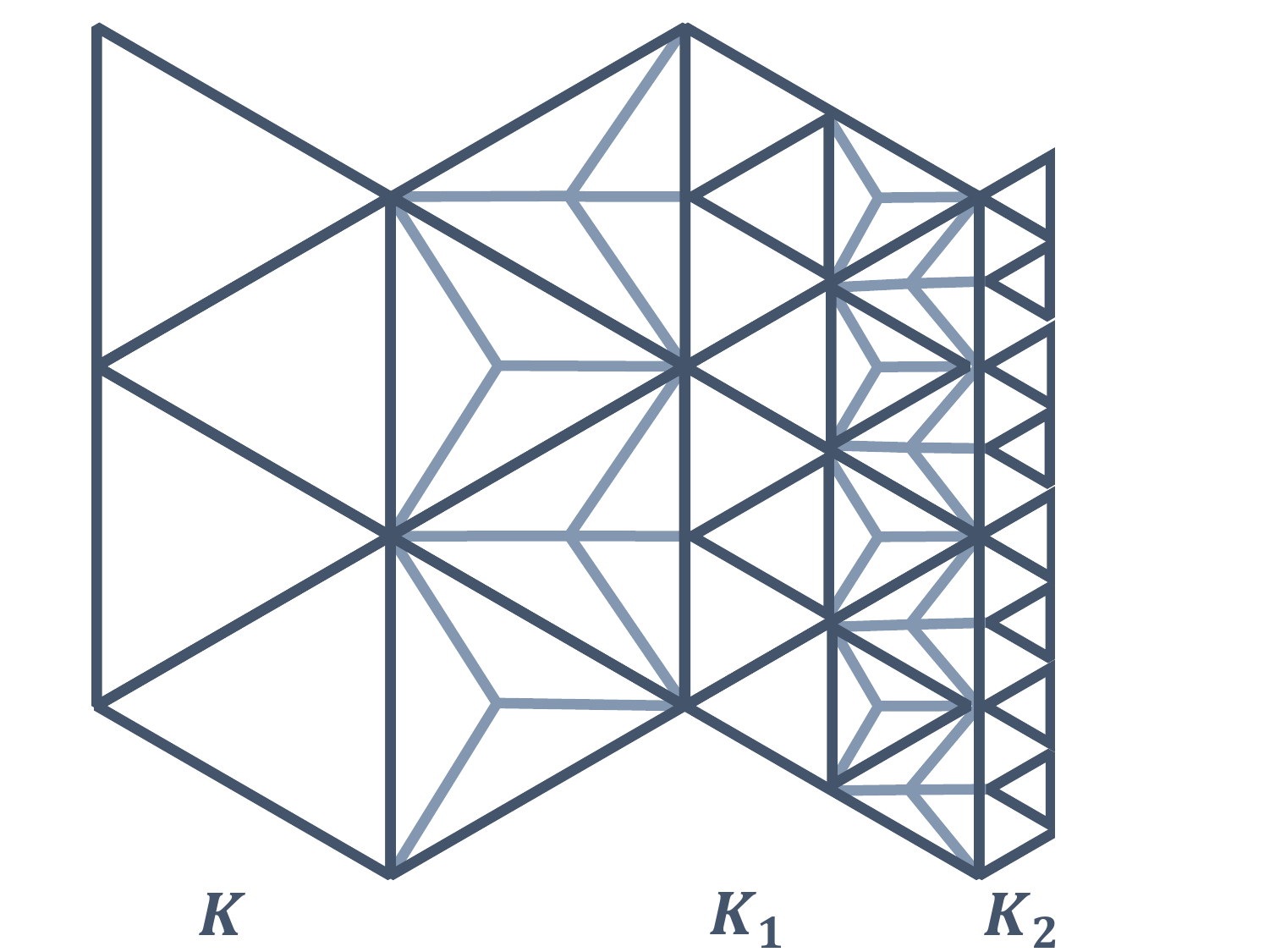}
	\centering
	\caption{A generalized crystalline subdivision of a polyhedron $K$. We have indicated below the crystalline subdivisions of $K$ and in-between we use the barycentric cone off which we have indicated in a lighter color. } \label{fig:gencrystsubdivT}
\end{figure}
\begin{proposition} \label{prop:crystallineConeOff}
    Let $K$ be a simplicial complex in $\R^N$, let $\delta>0$ be given and let $A \subset \R^n$ be a compact subset. Then, there exists an $L \in \N$ such that for all $\ell_0 \geq L$ there exist
    \begin{itemize}
        \item a finite collection $\SC$ of model simplices, and
        \item a natural number $C \in \N$,
    \end{itemize}
    such that for all $\ell_1 \geq \ell_0$ there exists a subdivision $K'$ of $K$ satisfying
    \begin{itemize}
        \item $K'$ agrees with $K_{\ell_0}$ on $A$,
        \item $K'$ agrees with $K_{\ell_1}$ on $\R^N \setminus B(A,\delta)$, 
        \item the simplices of $K' \cap B(A,\delta)$ are modeled by $\SC$, that is, they equal a model simplex in $\SC$ up to translation and scaling by $2^{\ell}$ where $\ell \in [\ell_0,\ell_1]$, and
        \item the subcomplex $K' \cap B(A,\delta)$ can be colored with at most $C$ colors.
    \end{itemize}
\end{proposition}
\begin{proof}
    \pfstep{Constructing the control data} 
    The results in \cref{sec:propCrystalline} yield a maximum number of colors $C'$ and a finite collection of model simplices $\SC'$ for all crystalline subdivisions of the subcomplex of $K$ consisting of simplices with non-empty intersection with $B(A,\delta)$.
    
    We enlarge the collection of model simplices $\SC'$ to account for the barycentric cone off of \cref{def:barycentricConeOff} as follows: Given a model simplex $\Delta \in \SC'$, we apply crystalline subdivision to subdivide one of its faces. Now using the barycentric cone off, we obtain a subdivision of $\Delta$ and we add the resulting simplices to $\SC'$. We repeat these steps for each face of $\Delta$ and each model simplex $\Delta \in \SC' $. This defines a finite collection $\SC$ of model simplices. 

    \pfstep{The construction} Choose $L$ to be large enough such that the simplices of $K_L \cap B(A,\delta)$ have a maximal radius of at most $ s= \delta/4$. Let natural numbers $\ell_1$ and $\ell_0$ be given such that $\ell_1 \geq \ell_0 \geq L$.
    
    Define $K^{0} \coloneq K$ and let $B^{(i)}$ denote a neighborhood of $A$ of size $\sum_{j=0}^i 2^{-j} s $. Then by induction we define a sequence of subdivisions $K^{(i)}$ as follows. Assume we have already constructed $K^{(i)}$, then we define its subdivision $K^{(i+1)}$ as the simplicial complex that\begin{itemize}
    	\item agrees with $K^{(i)}$ on $B^{(i)}$,
    	\item agrees with $K_{\ell_0+i+1}$ outside of $\ringT[K^{(i)}]{K^{(i)} \cap B^{(i)}}$, and
    	\item is the barycentric cone off on $\ringT[K^{(i)}]{K^{(i)} \cap B^{(i)}}$.
    \end{itemize} We note that the double star of $K^{(i)} \cap B^{(i)}$ is contained in $B^{(i+1)}$.

    We define $K' \coloneqq K^{(\ell_1-\ell_0)}$. We see that the coning off takes place in $B^{(\ell_1-\ell_0)}$, which is contained in neighborhood of $A$ of radius 
    \[\sum_{j=0}^{\ell_1-\ell_0} 2^{-j} s < \sum_{j=0}^\infty 2^{-j} s = 2s < \delta.\] 

    \pfstep{Wrapping up the proof} The claimed properties follow from the fact that the simplices in $K'$ are either simplices in $K_\ell$ for some $\ell \in [\ell_0,\ell_1]$, or are obtained from such a simplex by a single barycentric cone off. Hence they are indeed modeled by an element of $\SC$.
    
    We bound the maximum number of colors $C$ as follows. Each simplex is involved in at most one coning off process. Hence we can take $C = C' \cdot A$, where we defined $C'$ at the beginning of this proof, and $A$ is the resulting number of top-dimensional subsimplices after applying the barycentric cone off to a simplex $\Delta$. We can bound $A$ by $m-1+2^{m-1}$, where $m$ is the dimension of $\Delta$. To see this, let $F$ be the face of $\Delta$ subdivided once in a crystalline manner. Then the subdivision of $F$ contains $2^{m-1}$ simplices, which are each coned off to the barycenter. All the other faces of $\Delta$, of which there are $m-1$, are also coned off to the barycenter.
\end{proof}

We note that by the existence of the model simplices we get bounds on the minimal and maximal distance and their ratios, similar to \cref{lem:rmaxminBound}.  

We will refer to a subdivision $K'$ of $K$ as a \textbf{generalized crystalline subdivision}, if each simplex in $K'$ is an element of a crystalline subdivision of $K$ or is obtained from such a simplex by a single barycentric cone off where the faces are subdivided at most once in a crystalline manner. The subdivision $K'$ from \cref{prop:crystallineConeOff} is in particular an example, since each $K^{(i)}\cap B^{(i)}$ from the proof is a nice subcomplex. Hence we refer again to \cref{fig:gencrystsubdivT} for an example.

	\section{Linearization of piecewise maps} \label{sec:linearizing}

In this section we focus on the following notion:
\begin{definition} \label{def:linearization}
Let $K$ be a simplicial complex and let $|K|$ be the corresponding linear polyhedron. Let $s: |K| \to \R^n$ be a map that is piecewise smooth with respect to $K$. The \textbf{linearization} $s^\lin : |K| \to \R^n$ of $s$ with respect to $K$ is the unique piecewise linear map such that $s^\lin$ and $s$ agree on the vertices of $K$.
\end{definition}
We observe that $s$ and $s^\lin$ are homotopic through piecewise smooth maps, thanks to linear interpolation.

We now motivate this definition. We recall that our ultimate goal is to study sections of bundles $E$ over triangulated manifolds $T: |K| \rightarrow M$. In this setting we will not be able speak of linear maps into $E$. Nonetheless, we can consider linear maps with respect to specified fibered charts, but the transition functions between them will introduce non-linearity. Hence our strategy relies on the following observations:
\begin{itemize}
\item Linearizing a section can be done in a $C^1$-small manner if we subdivide $K$ enough.
\item Linearizing changes higher derivatives drastically, but this is of no consequence to us, since we are studying first order differential relations.
\item Once we have linearized, the jiggling argument is purely linear in nature.
\end{itemize}
As such, every time we move to a new fibered chart, we will linearize our data with respect to it.

\subsection{The linearization statement}

Fix a simplicial complex $K$ of pure dimension. Let $s: |K| \to \R^n$ be a map that is piecewise smooth with respect to $K$. We introduce the notation $s^\lin_\ell$ for the linearization of $s$ with respect to the $\ell$th crystalline subdivision $K_\ell$. Then:
\begin{lemma}\label{lem:linSecDistEucl}
Let $s: |K| \to \R^n$ be a piecewise smooth map with respect to a finite simplicial complex $K$ of pure dimension. Then, as $\ell \to \infty$ we have:
\begin{itemize}
\item $\dist0 (s,s^\lin_\ell) = O(2^{-2\ell})$, and
\item $\dist1 (s,s^\lin_\ell) = O(2^{-\ell})$.
\end{itemize} 
\end{lemma}
We recall from \cref{ssec:PSmaps} that the $C^0$ and $C^1$-distances used here are independent of $\ell$; they are computed using a fixed collection of metrics on jet spaces over the simplices of $K$.

\begin{proof} 
It suffices to prove the claim in the case where $K$ consists of one top-dimensional linear simplex, because we can repeat the following argument for every simplex, invoking finiteness of $K$. Moreover, after choosing an ordering on $K$, we can assume that we are working with the standard simplex $\Delta^m \subset \R^m$ up to an affine change of coordinates. We write $\Hom(\R^m,\R^n)$ for the space of linear maps $\R^m \rightarrow \R^n$ and we endow it with its standard Euclidean distance.

Consider an $m$-simplex $\Delta \in (\Delta^m)_\ell$. Let $\SD\subset \Hom(\R^m,\R^n)$ be the subspace of all derivatives $d_y s: \R^m \to \R^n$ with $y \in \Delta$. Then by Taylor's approximation formula, it holds that the maximum distance $d_\mathrm{max}$ between two elements in $\SD$ is $O(\rmax(\Delta))$. This maximum exists because $\Delta$ and hence $\SD$ is compact. We note that we obtain in particular for all $i \in \{1,\dots,m\}$ and $y_1,y_2 \in \Delta$ that
\[|(d_{y_1}s_j) e_i - (d_{y_2}s_j) e_i | \leq |(d_{y_1}s_j)  - (d_{y_2}s_j)  | \leq d_\mathrm{max}. \]

We also fix a vertex $v \in \Delta$ and $e_1,\dots e_m \in \R^m$ such that the other vertices of $\Delta$ are $\{ v + e_i \}_i$. Then $s_\ell^\lin|_\Delta$ is the map that sends $x\in \Delta$ to $s(v) + A(x-v)$ where $A$ is the matrix with columns $s(v+e_i)-s(v)$. By the mean value theorem there exist elements $y_{ij} \in \Delta$ such that $A_{ij} = (d_{y_{ij}}s_j) e_i$. Hence we see that for the $j$th row of $A$ we obtain $A_j x = (d_v s_j) x + O(d_\mathrm{max}|x|)$ for any $x\in \R^m$.

For the $C^0$-distance between $s_\ell^\lin$ and $s$ we see that for their $j$th components we have 
\begin{align*}
    \dist0 (s_j(x),(s_\ell^\lin)_j(x)) & = |s_j(x) -s_j(v) -A_j(x-v)| \\
    &\leq |s_j(x) - s_j(v) - (d_v s_j)(x-v) | + O(d_\mathrm{max}|x-v|)    
\end{align*}
which we bound as $O(\rmax^2(\Delta))$, applying Taylor's approximation to the first term.

For the $C^1$-distance, we consider the matrices $d_x s$ and $d_x s^\lin_\ell$ for some $x \in \Delta$, where we point out that the latter is in fact the matrix $A$ from earlier. The difference between their $(ij)$th elements is $(d_x s_j - d_{y_{ij}}s_j) e_i$, which we bound using Taylor's approximation by $O(|x-y_{ij}|) = O(\rmax(\Delta))$.

We now conclude the proof by recalling from \cref{lem:rmaxminBound} that $\rmax(\Delta) = O(2^{-\ell})$.
\end{proof}
We observe that the proof presented does not rely specifically on the fact that we are subdividing in a crystalline fashion. The only property it uses is that the diameter $\rmax$ of the simplices goes to zero as $\ell$ goes to infinity from \cref{lem:rmaxminBound}.

\subsection{Modifying maps on a face} \label{ssec:relativeLinearization}

As explained above, we linearize in order to work locally in a linear manner. This means that we need to explain how our local arguments globalize. To this end, we introduce an interpolation procedure between sections.
\begin{definition} \label{def:interpol}
Let $\Delta$ be a linear simplex given as the join of two opposing faces $A$ and $B$. Let $t: \Delta \to [0,1]$ be the join parameter: the affine function that is $0$ over $A$ and $1$ over $B$. Suppose that $s_A,s_B : \Delta \to \R^n$ are two smooth maps. We define the \textbf{interpolation} of $s_A$ and $s_B$ over $\Delta$ with respect to $A$ and $B$, denoted as $\interpolate{\Delta}{A,B}{s_A}{s_B}$, to be the map $x \mapsto t(x) s_A(x) + (1-t(x)) s_B(x)$.
\end{definition}

The result of interpolation between sections is controlled by the shape of the simplex and the difference between the two sections that serve as input:
\begin{lemma}\label{lem:interpol}
Fix a linear $m$-simplex $\Delta$ spanned by two opposing faces $A, B$. Consider moreover three maps $u_1,u_2: A \rightarrow \R^n$, $v: B \to \R^n$, and the interpolations $s_i = \interpolate{\Delta}{A,B}{u_i}{v}$. Then the following bounds hold:
\begin{itemize}
	\item $\dist{0}(s_1,s_2) = O(\dist0 (u_1,u_2))$, and
	\item $\dist{1}(s_1,s_2) = \dist{0}(s_1,s_2) + O\left(\frac{\dist0 (u_1,u_2)}{\rmin(\Delta)} + \dist1 (u_1,u_2) \right)$.
\end{itemize}
\end{lemma}
Here we can use the usual Euclidean $C^0$ and $C^1$-distances, although any other choices are equivalent up to a constant.
\begin{proof}
The statement regarding the $C^0$-distance is clear. For the $C^1$-claim we argue as follows: The join parameter $t:\Delta \rightarrow [0,1]$ is a smooth function whose partial derivatives are bounded above by $1/\rmin(\Delta)$. The interpolation sections $s_i$ can be written as $t u_i + (1-t) v$. Hence we obtain:
\begin{align*}
\left| \partial_j (s_1 -s_2)  \right|
& = \left| \partial_j \left( t \cdot u_1 - t \cdot u_2 \right) \right| \\
& \leq \left| \partial_j(t) \right| \left|u_1  - u_2  \right| +   |t|\left| \partial_j (u_1 - u_2 ) \right| \\
&\leq \frac{1}{\rmin(\Delta)}  \dist0 (u_1,u_2) + \dist1 (u_1,u_2) . \qedhere 
\end{align*}
\end{proof}

\subsubsection{The linear case}
We note that when we interpolate sections that are linear over a simplex $\Delta$, the result is in general not linear. Hence we define the following variation of \cref{def:interpol}, which only depends on the sections restricted to the faces.

\begin{definition} \label{def:join}
    Let $\Delta$ be a linear simplex given as the join of two opposing faces $A$ and $B$. Suppose that $s_A: A \to \R^n$ and $s_B : B \to \R^n$ are two affine maps. We define the \textbf{join} of $s_A$ and $s_B$ over $\Delta$, denoted as $\joins{\Delta}{A,B}{s_A}{s_B}$, to be the unique affine map $\Delta \to \R^n$ agreeing with $s_A$ over $A$ and agreeing with $s_B$ over $B$.
\end{definition}
For the join we now obtain slightly simpler bounds when perturbing, than for the interpolation in \cref{lem:interpol}: 
\begin{lemma}\label{lem:join}
Fix a linear $m$-simplex $\Delta$ spanned by two opposing faces $A, B$. Consider moreover the linear maps $v: A \to \R^n$ and $u_1,u_2: B \rightarrow \R^n$, and the corresponding joins $s_i = \joins{\Delta}{A,B}{v}{u_i}$. Then the following bounds hold:
\begin{itemize}
	\item $\dist{0}(s_1,s_2) = O(\dist0 (u_1,u_2))$, and
	\item $\dist{1}(s_1,s_2) = \dist{0}(s_1,s_2) + O \left( \dist0 (u_1,u_2) \cdot \maxcoeff(\Delta) \right)$.
\end{itemize}
\end{lemma}
\begin{proof}
    The statement regarding the $C^0$-distance is again clear, so we only argue the $C^1$-claim. We label the vertices of $\Delta$ by $v_0,\dots,v_m$, and assume without loss of generality that $v_0=0$. Additionally we assume that $A$ is spanned by $v_0,\dots,v_\ell$ for some $\ell < m$.
    
    We define affine sections $\tilde s_i : \Delta^m \to \R^n$ by requiring that $\tilde s_i(0)=s_i(0)$ and $\tilde s_i(e_i) = s_i(v_i)$, where the vectors $e_i$ are the unit vectors in $\R^m$. If we let $T$ be the linear transformation $\Delta \to \Delta^m$ sending $v_i$ to $e_i$, we see that $s_i = \tilde s_i \circ T$. We recall from \cref{lem:rminmaxOperator} that the operator norm of $T$ is $O (\maxcoeff(\Delta))$.

    We denote the derivative of $\tilde s_i$ by the matrix $D_i$. The $j$th column of $D_i$ equals $v(v_j)-v(0)$ if $j \in\{1,\dots,\ell\}$ and $u_i(v_j)-v(0)$ if $j \in\{\ell+1,\dots,m\}$. The $j$th column of the matrix $D_1-D_2$ consists of zero for the first $\ell$ columns, and $u_1(v_j)-u_2(v_j)$ if $j \in\{\ell+1,\dots,m\}$. Hence the difference $|D_1-D_2|$ can be bounded as $O(\dist0 (u_1,u_2))$.

    We now revert our attention back to the $s_i$ and obtain:
    \begin{align*}
        \dist0 (ds_1,ds_2) &= O \left(\lVert D_1 T - D_2 T \rVert\right) \\
         &\leq O \left(\lVert D_1 -D_2\rVert \lVert T \rVert\right) 
        = O\left (\dist0 (u_1,u_2) \cdot \maxcoeff(\Delta) \right ). \qedhere
    \end{align*}
\end{proof}

\subsection{Relative linearization}

We now put together \cref{lem:interpol,lem:linSecDistEucl} in order to yield a linearization statement that applies over a subcomplex, relative to the complement of small neighborhood. We point out that in the proof below, it is the niceness of the subcomplexes (recall \cref{def:niceSubcomplex}), that allows us on the ring of the nice subcomplexes to interpolate between sections.
\begin{corollary}\label{cor:linSecDist}
Let $K$ be a simplicial complex in $\R^N$. Let $s: |K| \to \R^n$ be a map that is piecewise smooth with respect to $K$. Let $K' \subset K$ be a finite subcomplex and consider a neighborhood $\Op(|K'|)$.

Then, for each sufficiently large $\ell$ there exists a map $s_\ell': |K| \to \R^n$ that is piecewise smooth with respect to $K_\ell$ and moreover:
	\begin{itemize}
		\item $s_\ell'$ is piecewise linear over $|K'|$,
		\item $s_\ell' = s$ outside of $\Op(|K'|)$,
		\item $\dist{0}(s_\ell',s) = O(2^{-2\ell})$, and
		\item $\dist{1}(s_\ell',s) = O(2^{-\ell})$.
	\end{itemize}
\end{corollary}
\begin{proof}
Given an open neighborhood $\Op(|K'|)$ we can find a finite collection of convex and open subsets $\{U_i\}_{i=1}^{I}$ of $\R^N$ so that $\cup_i U_i $ intersected with $ |K|$ is a neighborhood of $|K'|$ contained in $\Op(|K'|)$. We fix such a collection and assume that $\ell$ is large enough so that each simplex in $K_\ell$ is either contained in one of the $U_i$ or is disjoint from $|K'|$. By \cref{lem:niceSubcomplex} it follows that each subcomplex $K^i \coloneqq U_i \cap K$ is nice.

We now construct maps $s_i : |K| \to \R^n$ by induction on $i$ which each satisfy all of the requirements of the statement, as long as we restrict $K'$ to $\cup_{j=0}^{i} U_j \cap K'$. We set $s_0 := s$, so that the base case $i=1$ follows the same proof as the inductive step. Hence assume $s_{i-1}$ is constructed. We use \cref{lem:linSecDistEucl} to linearize $s_{i-1}$ over $|K^i|$, yielding some map $\overline{s}_i : |K^i| \to \R^n$. Over the ring of $K^i$ we interpolate between $\overline{s}_i$ and $s_{i-1}$; we call this interpolation $s_i$ and bound its difference to $s_{i-1}$ over the ring of $K^i$ using \cref{lem:interpol}. We define $s'_\ell$ as the end result of the induction, that is $s'_\ell = s_I$.

We observe that the map $s'_\ell$ is piecewise smooth: Smoothness on each simplex is clear from the interpolation formula. Continuity over $K$ follows because the join parameter $t$ is continuous over the star of each $K^i$ and the interpolation over two adjacent simplices agrees on their intersection.

To conclude we observe that, at each stage $i$, the claimed bounds hold for $s_i$ by combining \cref{lem:linSecDistEucl} (on $K^j$) and \cref{lem:interpol} (on $\str(K^i)$). As such, the bounds apply to $s'_\ell$, completing the proof.
\end{proof}
We will refer to $s_\ell'$ from \cref{cor:linSecDist} as a linearization of $s$ over $|K'|$.

\section{Slope perturbations of piecewise linear maps} \label{sec:perturbing}

Let $E \rightarrow M$ be a fiber bundle over a smooth triangulated manifold $M$. We are going to explain next the idea of perturbing a section $s: M \rightarrow E$ by tilting it slightly simplex by simplex. Doing this with some quantitative control is the other main ingredient behind jiggling (the first ingredient being crystalline subdivision).

First, in \cref{ssec:extending}, we discuss how to approximate a $1$-jet $\sigma \in \pi_f^{-1}(x) \subset J^1(E)$ by a locally defined holonomic solution of a given relation $\SR \subset J^1(E)$, where we recall that $\pi_f: J^1(E)\to E$ is the front projection. In order to do this, we need the relation to be fiberwise dense (\cref{def:fiberweiseDense}) and open.

This idea is then applied in \cref{ssec:perturbing} to construct a solution over an individual simplex. This solution must then be joined to the unperturbed map away from the simplex. This will allow us to introduce such tilts colorwise, in an estimated manner.

Even though our ultimate goal is to make statements about sections of bundles over manifolds, we continue working (mostly) in the linear setting. We will explain in detail in \cref{sec:jiggling} how one makes the transition to sections defined on manifolds, but the main idea is to work ``chart by chart''.

\subsection{Linear extensions of jets} \label{ssec:extending}

To approximate a $1$-jet $\sigma \in \pi_f^{-1}(x) \subset J^1(E)$ by a local solution of $\SR$, we use a two step process. We first change the slope of $\sigma$ such that it lies in the relation, then we extend it to a holonomic section of $\SR$. We emphasize that in the first step we do not change the zeroth order information of the jet, only its slope, so that the jet remains in the fiber $\pi_f^{-1}(x)$.

\subsubsection{Density of solution jets}

We now fix metrics on $M$, $E$, and $J^1(E)$. We use the notation $B_f(x,r)$ for a ball around $x \in J^1(E)$ of radius $r$ intersected with the $\pif$-fiber through $x$, that is $B_f(x,r) = B(x,r) \cap \pif^{-1}(\pif(x))$. The two requirements we pose on $\SR$ immediately imply the following, which realizes the first step:
\begin{lemma} \label{lem:sigmaBall}
Let $\SR \subset J^1(E)$ be an open and fiberwise dense relation. Then, for all $\epsilon>0$ and $\sigma \in J^1(E)$, there exists $\sigma'\in B_f(\sigma, \epsilon)$ and $\delta>0$ such that $B(\sigma',\delta) \subset \SR$.
\end{lemma}

In fact, if we restrict ourselves to $\sigma$ in a compact set and use that $\delta$ depends continuously on $\sigma$, we
see that $\delta$ can be chosen such that it does no longer depend on $\sigma$. This leads to the following:
\begin{corollary} \label{cor:ballLinExtAff}
Let $\SR \subset J^1(E)$ be an open and fiberwise dense relation and let $A \subset J^1(E)$ be a compact subset. For all $\epsilon>0$ there exists $\delta>0$ such that: For every $\sigma \in A$, there exists $\sigma'\in B_f(\sigma, \epsilon)$ with $B(\sigma',\delta) \subset \SR$.
\end{corollary}

\subsubsection{Density of linear extensions}

To extend the jets in $B(\sigma',\delta)$ to holonomic sections of $\SR$, we make use of linear extensions. In order to make sense of linear, we have to assume here that $M = \R^m$ and $E$ is the trivial bundle $\trivbr$.
\begin{definition}
Let $U \subset \R^m$ be a subset. We say that $s: U \to \R^n$ is the \textbf{linear extension} of $\sigma \in J^1(\R^m,\R^n)$ if $j^1 s (\pi(\sigma)) = \sigma$ and $s$ is linear. We write in this case $s = \lext(\sigma)$ or, when we want to emphasize the domain, $ s= \lextD[U](\sigma)$.
\end{definition}

With this definition we see that \cref{lem:sigmaBall} tells us that linear extensions of the jets in $B(\sigma',\delta)$ are actually local solutions. Here we write $\pi$ for the projection $J^1(\R^m,\R^n) \to \R^m$.
\begin{corollary} \label{lem:ballLinExt}
Let $\SR \subset J^1(\R^m,\R^n)$ be an open and fiberwise dense relation, and let $\epsilon>0$ and $\sigma \in J^1(\R^m,\R^n)$ be given. Then there exists $\sigma' \in B_f(\sigma, \epsilon)$ and $\delta >0$ such that for all $\sigma'' \in B(\sigma', \delta)$ we have that $\lext (\sigma'')$ is a solution of $\SR$ over $B(\pi (\sigma''), \delta) \subset \R^m$.
\end{corollary}

As in \cref{cor:ballLinExtAff}, we now obtain:
\begin{corollary}\label{cor:ballLinExt}
Let $\SR \subset J^1(\R^m,\R^n)$ be an open and fiberwise dense relation, let $\epsilon>0$ be given and consider a compact subset $A\subset J^1(\R^m,\R^n)$. Then there exists $\delta>0$ such that: for all $\sigma \in A$ there exists $\sigma' \in B_f(\sigma, \epsilon)$ such that for all $\sigma'' \in B(\sigma', \delta)$ we have that $\lext (\sigma'')$ is a solution of $\SR$ over $B(\pi (\sigma''), \delta)$.
\end{corollary}

\subsection{Slope-perturbing} \label{ssec:perturbing}

We now use the linear extensions of the previous subsection to tilt maps in a piecewise manner. The upcoming statements are quantitative and use the standard $C^r$-distance given by the ambient Euclidean space the simplices (or simplicial complexes) live in. By compactness of simplices, this is equivalent to any other $C^r$-distance we could choose.

\subsubsection{Over a single simplex}

Consider a linear simplex $\Delta$. The reasoning from the previous subsection states that, given $\epsilon>0$ and a $1$-jet $\sigma_1 \in J^1(\Delta,\R^n)$, we can choose $\sigma_2 \in B_f(\sigma_1,\epsilon)$ and consider its linear extension $s_1$ over $\Delta$. Naturally, if we had chosen some other $\sigma_2 \in B_f(\sigma,\epsilon)$ instead of $\sigma_1$, we would have obtained a different linear extension $s_2$. However, the difference between $s_1$ and $s_2$ is bounded by $\epsilon$ and the size of $\Delta$. The following is then immediate.
\begin{lemma}\label{lem:perturbSlopeToPoints}
Consider
\begin{itemize}
    \item a positive number $\epsilon$,
    \item a linear simplex $\Delta \subset \R^N$,
    \item a pair of jets $\sigma_1,\sigma_2 \in J^1(\Delta,\R^n)$ satisfying $\sigma_2 \in B_f(\sigma_1,\epsilon)$, and
    \item their linear extensions $s_i= \lextD(\sigma_i)$.
\end{itemize}
Then the following bounds hold:
\begin{itemize}
    \item $\dist{0}(s_1,s_2) = O(\epsilon\rmax(\Delta))$, and
    \item $\dist{1}(s_1,s_2) = O\left(\epsilon + \epsilon\rmax(\Delta)\right)$.
\end{itemize}
\end{lemma}
We observe that the second claim would still hold if $\sigma_1$ and $\sigma_2$ had lain in the same \emph{fiber} over $x \in M$. The first claim would not, as we then have that $\dist0(\sigma_1,\sigma_2) = O(\epsilon+\epsilon\rmax(\Delta))$.  To emphasize that $\sigma_1$ and $\sigma_2$ lie in the same \emph{front fiber}, we refer to $s_2$ as an $\epsilon$-\textbf{slope-perturbation} of $s_1$.

\subsubsection{Over a color}

We now extend this procedure to all the simplices of a given color. We package this into a definition:
\begin{definition} \label{def:perturbSlopeColor}
Let $\epsilon>0$ and let $K$ be a finite, colored, simplicial complex of pure dimension. Fix a piecewise linear map $s : |K| \to \R^n$ and a color $i$. We say that $s': |K| \to \R^n$ is the result of \textbf{slope-perturbing $s$ by $\epsilon$ over the $i$th color} if $s'$ is obtained from $s$ by
\begin{itemize}
    \item first $\epsilon$-slope-perturbing over all simplices of color $i$ (as in Lemma \ref{lem:perturbSlopeToPoints}), and 
    \item then applying the join construction (from Definition \ref{def:join}) in the star of all such simplices.
\end{itemize}
\end{definition}
We note that $s'$ is in particular piecewise linear and have depicted $s'$ in \cref{fig:slopePerturb}.

\begin{figure}[h]
    \includegraphics[width=0.7\textwidth,page=7, clip=true, trim = 0 4.5cm 0 7cm]{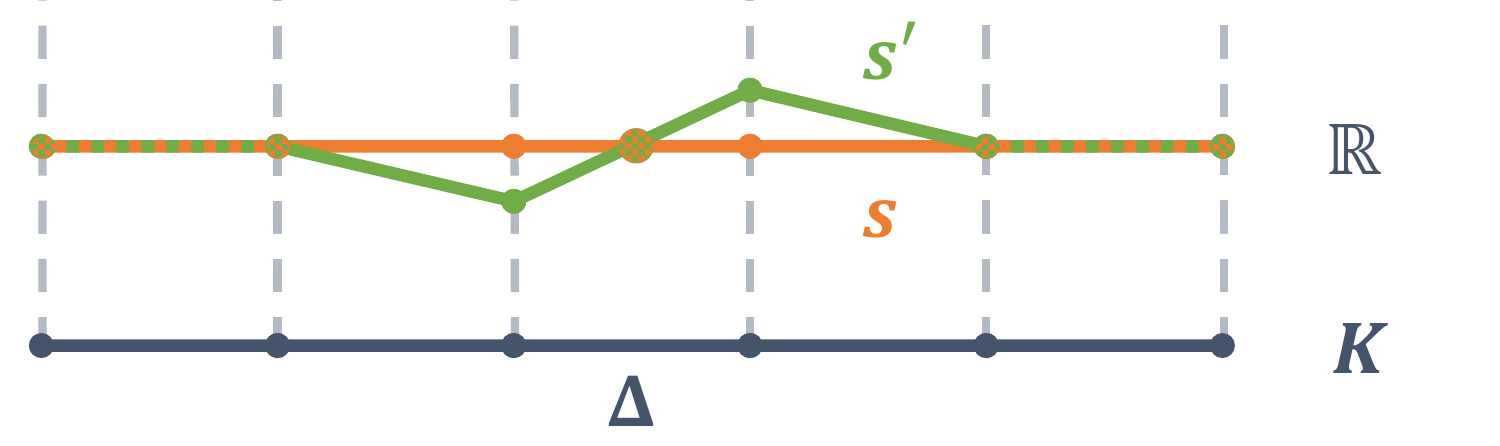} 
    \centering
	\caption{An example of the result $s'$ of slope-perturbing $s$ over a simplex $\Delta$ in $K$. We use the join over $\str(\Delta^1)$.
    } \label{fig:slopePerturb}
\end{figure}

The following statement is immediate from the bounds established in \cref{lem:perturbSlopeToPoints,lem:join}:
\begin{lemma} \label{lem:perturbSlopeColor}
Let $K$ be a colored simplicial complex of pure dimension $m$. Let $K'$ be the collection of simplices of color $i$, assume it is finite. Suppose $s': |K| \to \R^n$ is the result of slope-perturbing a piecewise linear map $s : |K| \to \R^n$ by $\epsilon$ over the $i$th color. Let $\rmax$ and $\maxcoeff$ be the maximum value of respectively $\rmax(\Delta)$ and $\maxcoeff(\Delta)$ over $\str(K')$ and let $\rmin$ be the minimum value of $\rmin(\Delta)$ over $\str(K')$. Then, the following bounds hold:
\begin{itemize}
    \item $\dist{0}( s, s') = O(\epsilon\rmax)$, and
    \item $\dist{1}( s, s') = O\left(\epsilon\rmax \left( 1 + \maxcoeff \right)\right)$. \qedhere
\end{itemize}
\end{lemma}
	\section{Jiggling} \label{sec:jiggling}

In this section we establish our main jiggling results. We first tackle jiggling in the linear setting (\cref{sec:jigglingEucl}), for maps of linear polyhedra into Euclidean space. We then address jiggling for sections of bundles over manifolds (\cref{sec:jigglingMfd}).

\subsection{Jiggling in the linear setting} \label{sec:jigglingEucl}

We can formulate our first jiggling theorem, using the terminology for differential relations over polyhedra from \cref{sec:JetPolyhedra}. 
\begin{theorem} \label{th:jigglingEucl}
Consider a finite simplicial complex $K$ of pure dimension, and an open and fiberwise dense differential relation $\SR := \{\SR_\Delta\}_{\Delta \in K^{(top)}}$, where $\SR_\Delta \subset J^1(\Delta,\R^n)$ for all $\Delta \in K^{(top)}$. Then, given
	\begin{itemize}
		\item $\epsilon>0$, and
		\item a piecewise smooth map $s: |K| \to \R^n$,
	\end{itemize}
there exists an $\epsilon$-jiggling $(s',K_\ell)$ of $(s,K)$ such that
\begin{itemize}
	\item $s'$ is piecewise linear with respect to the crystalline subdivision $K_\ell$ of $K$, and
	\item $s'$ is a piecewise solution of $\SR$.
\end{itemize}
\end{theorem}
\begin{proof}
We write $C$ for the number of colors needed to color all the $K_\ell$ (\cref{lem:coloringCrystalline}). We write $K_\ell^{(\topd,i)}$ for the collection of top-dimensional simplices in $K_\ell$ of color $i$.

\pfstep{Strategy of the proof} The strategy we employ is the following: we linearize $s$ with respect to $K_\ell$ and then, inductively on the color $i \in [C]$, we slope-perturb by a constant $\epsilon_i$ so that the resulting section $s_\ell^{(i)}$ is a piecewise smooth solution of $\SR$ over the simplices of color $j \leq i$. The key points to pay attention to are:
\begin{itemize}
\item $\ell$ has to be large enough so that the linearization is a good approximation.
\item $\ell$ has to be large enough to guarantee that solutions can be constructed via slope-perturbation.
\item each $j^1 s_\ell^{(i)}$ has a $\delta_i$ neighborhood lying in $\SR$, which determines how large future perturbations and hence $\epsilon_{i+1}$ can be, so that the perturbed section is still a solution. 
\end{itemize}

\pfstep{Choosing a good collection of constants} We now choose the sequences $\{ \epsilon_i \}_{i \in [C]}$ and $\{ \delta_i \}_{i \in [C]}$ and $\ell$. A first requirement is that $\ell$ must be large enough such that $\dist1 (s, s_\ell^{\lin}) < \epsilon/2$, which can be achieved using \cref{lem:linSecDistEucl}.

We then claim that there exist sequences $\{\epsilon_i\}_{i \in [C]}$ and $\{\delta_i\}_{i \in [C]}$ such that
\begin{enumerate}[label={ (\arabic*)}]
	\item\label{it:epsilon} an $\epsilon_i$-slope-perturbation, of any piecewise linear map $h : |K| \to \R^n$ that is $\epsilon$-close to $s$ in the $C^1$-distance, over all simplices of the same color, causes at most a $\delta_{i-1}/4$ perturbation over adjacent simplices, in a $C^1$-sense, 
	\item\label{it:delta} any $1$-jet $\sigma \in \overline{ B(j^1 s(|K|), \epsilon_i)}$ admits an $\epsilon_i$-slope-perturbation $\sigma'$ satisfying that the $\delta_i$-ball around it lies in $\SR$. Moreover, the linear extension of each $\sigma''\in B(\sigma',\delta_i)$ over the $\delta_i$-ball around $\pi(\sigma'')$ is a solution of $\SR$,
	\item $\epsilon_{i+1} < \epsilon_i$, and
	\item $2 \delta_{i+1} < \delta_i$.
\end{enumerate}
These sequences can be constructed inductively on $i \in [C]$. When we set $\delta_{-1} = \frac{\epsilon}{2C}$ and $\epsilon_{-1} = 1$, we can treat the base case $i=0$ as we treat the inductive step. Hence we assume that we have constructed $\delta_{i-1}$ and $\epsilon_{i-1}$. For the existence of $\epsilon_i$, two comments are relevant: First we note that we can find $\epsilon_i$ such that \cref{it:epsilon} holds for $s^\lin_\ell$ because according to \cref{lem:rmaxminBound}, we can bound the size of any simplex in $K_\ell$ in such a way that we can bound the $C^1$-distance independently of $\ell$. Next, we point out that the $1$-jet of $h$ is contained in the $\epsilon$-neighborhood of $j^1s$ in $J^1(|K|,\R^n)$, which is compact, which allows us to find an $\epsilon_0$ that works for all such $h$. The constants $\delta_i$ such that \cref{it:delta} holds, are produced by \cref{cor:ballLinExt}.

Lastly, we ask $\ell$ to be large enough so that $\max(1,||s||_{C^1} + 2\epsilon)\rmax(\Delta) < \delta_C/2$ for every $\Delta \in K_\ell$. That this is possible follows from \cref{lem:rmaxminBound}.
	
\pfstep{Induction hypothesis} We can now construct the requested jiggling $s'$ by induction on $i \in [C]$. The induction hypothesis is that there exists a continuous and piecewise linear map $s_\ell^{(i)} : |K| \to \R^n$ such that 
	\begin{enumerate}[label={ (IH\arabic*)}]
		\item \label{item:IHsolNbh} for all $\Delta \in \cup_{j=0}^i K_\ell^{(\topd,j)}$, all $x\in \Delta$ and $\sigma'' \in B(j^1 s_\ell^{(i)} |_{\Delta} (x) , \delta_i/2) $, the linear extension of $\sigma''$ over $\Delta$ is a solution of $\SR$, and
		\item \label{item:IHdist} $\dist1 (s_\ell^{(i-1)}, s_\ell^{(i)}) < \frac{\epsilon}{2C}$ where $s_\ell^{(-1)} = s^{\lin}_\ell$.
	\end{enumerate} 
Since $\cup_{j=0}^{-1} K_\ell^{(\topd,j)} = \emptyset$, the base case follows from the induction step and hence we only prove the latter. Consider therefore $s_\ell^{(i)}$ as given. We want to construct $s_\ell^{(i+1)}$.
	
\pfstep{Construction} We construct $s_\ell^{(i+1)}$ by slope-perturbing $s_\ell^{(i)}$ by $\epsilon_{i+1}$ over all simplices in $K_\ell^{(\topd,i+1)}$, as follows. Consider a simplex $\Delta \in K_\ell^{(\topd,i+1)}$ with a vertex $v$. By definition of $\delta_{i+1}$, we know there exists $\sigma' \in B_f(j^1 s^{(i)}_\ell(v),\epsilon_{i+1}) $ such that for all $\sigma'' \in B(\sigma',\delta_{i+1})$ we have that $\lext(\sigma'')$ over $B(\pi(\sigma''),\delta_{i+1})$ is a solution. According to our assumptions on $\ell$, we have that $\rmax (\Delta) <\delta_C/2$. As such, we can define $s_\ell^{(i+1)}|_{\Delta}$ as $\lext{(\sigma')}$. According to \cref{lem:perturbSlopeColor}, we can do this for all simplices $\Delta \in K_\ell^{(\topd,i+1)}$ at the same time, applying the join construction over the star.
	
By our definition of $\{\epsilon_i\}_{i \in [C]}$, \cref{item:IHdist} is immediately satisfied for $s_\ell^{(i)}$ and $s_\ell^{(+1i)}$. To see that \cref{item:IHsolNbh} is satisfied on each simplex $\Delta \in \cup_{j=0}^i K_\ell^{(i+1)}$, we distinguish two cases: $\Delta$ is of color $i+1$ or of a previous color.
	
\pfstep{Simplices of color $i+1$} Take $\sigma'' \in B(j^1 s_\ell^{(i+1)} |_{\Delta} (x) , \delta_{i+1}/2) $ for some $x \in \Delta$. Recall that $v$ is the vertex of $\Delta$ over which we chose the $1$-jet $\sigma' $ we linearly extended and that hence we have $ j^1 s^{(i+1)}_\ell(v) = \sigma'$. Then we see that
	\begin{align*}
		\dist{0}(\sigma'', j^1 s^{(i+1)}_\ell(v))
		&\leq \dist{0}(\sigma'', j^1 s^{(i+1)}_\ell(x)) + \dist{0}(j^1 s^{(i+1)}_\ell(x), j^1 s^{(i+1)}_\ell(v)) \\
		&< \delta_{i+1}/2 + \dist{0}(j^1 s^{(i+1)}_\ell(x), j^1 s^{(i+1)}_\ell(v)),
	\end{align*}
where the latter term is bounded by $||s^{(i+1)}_\ell||_{C^1}\rmax(\Delta)$. Our assumptions on $\ell$ thus imply that $\sigma'' \in B(j^1 s^{(i+1)}_\ell(v),  \delta_{i+1})$. Hence, the section $\lext(\sigma'')$ is a solution over $\Delta$.  
	
\pfstep{Simplices of a previous color} Let now $\Delta$ be in $K_\ell^{(\topd,j)}$ for some $j<1+i$ and take take $\sigma'' \in B(j^1 s_\ell^{(i+1)}(x), \delta_{i+1}/2)$ with $x\in \Delta$. By definition of $\epsilon_{i+1}$ it follows in particular that we have $\dist{0}(j^1s_\ell^{(i)} , j^1s_\ell^{(i+1)} ) < \delta_{i}/2$, and hence we see that 
	\begin{align*}
		\dist{0} ( \sigma'', j^1 s_\ell^{(i)}(x)) 
		&\leq \dist{0} ( \sigma'', j^1 s_\ell^{(i+1)}(x) ) + \dist{0} ( j^1 s_\ell^{(i+1)}(x), j^1 s_\ell^{(i)}(x) ) \\
		&< \delta_{i+1}/2 + \delta_{i}/4 < \delta_i/2,
	\end{align*}
    where the last step follows since we assumed that $2 \delta_{i+1} < \delta_i$. Therefore $\lext(\sigma'')$ over $\Delta$ is a solution of $\SR$ and \cref{item:IHsolNbh} is still satisfied over $\Delta$.
	
	\pfstep{End of induction} We define $s'= s_\ell^{(C)}$. The inductive hypotheses imply that $s'$ is a piecewise linear solution approximating $s$.
\end{proof}

With minimal adaptations, the proof we have given also provides us with the relative case. We will refer to $s'$ below as the result of jiggling $s$ relative to $|K'|$ and $|K''|$.

\label{sec:jigglingEuclRel}
\begin{corollary}\label{cor:jigglingEuclRel}
Suppose $K$, $\SR$, $\epsilon$, and $s$ are given as in \cref{th:jigglingEucl}. Suppose moreover that we are given disjoint subcomplexes $K',K'' \subset K$ (of full dimension) with disjoint neighborhoods $\Op(|K'|),\Op(|K''|) \subset |K|$ such that $s|_{\Op(|K'|)}$ is a solution of $\SR$. Then, there exists an $\epsilon$-jiggling $(s': |K| \to \R^n, K_\ell)$ such that
\begin{itemize}
    \item $s'$ is piecewise linear on $|K|\setminus \left( \Op(|K'|) \cup \Op(|K''|) \right)$ with respect to $K_\ell$,
	\item $s'$ is a piecewise smooth solution of $\SR$ on $|K|\setminus\Op(|K''|)$, and
	\item $s' = s$ on $|K' \cup K''|$.
\end{itemize}
\end{corollary}
\begin{proof}
Let $L \in \N$ be such that $\left|\strN{2}(K'\cup K'',K_L)\right| \subset \Op(|K'|) \cup \Op(|K''|)$, implying that $s$ restricted to $\left|\strN{2}(K',K_L)\right|$ is a solution. Then since $K'$ is finite, there exists $\delta>0$ such that
\[ B \left( j^1 s |_{\left| \strN{2}(K',K_L) \right|} ,\delta \right) \subset \SR. \]
We take the number of subdivisions $\ell$ to be at least $L$.

Let $\{U_i\}_{i \in I}$ be a finite collection of convex opens in $\R^N$ so that their union $\cup_i U_i$ intersected with $|K|$ covers $M \setminus (\Op(|K'|) \cup \Op(|K''|))$, but is disjoint from $|K'\cup K''|$. We define $K^i = K_\ell \cap U_i$ and note that they are nice subcomplexes of $K$ so that we can interpolate over their ring as in \cref{def:interpol}. We assume that $\ell$ is large enough so that $\cup_i |K^i|$ covers $M \setminus (\Op(|K'|) \cup \Op(|K''|))$ and so that $\str(K^i) \cap (K' \cup K'')= \emptyset$. 

Define $s^{(-1)} = s$. Then by induction on $i$ we can construct maps $s^{(i)}: |K| \to \R^n$ which are a solution over $\cup_{j=0}^i K^j$. We do so by applying the proof of \cref{th:jigglingEucl} to $s^{(i-1)}$ restricted to $K^i$, that is: we color the simplices in $K^i$ and slope-perturb over those simplices. When subdivision is needed for jiggling, we subdivide the entirety of $K$. We define the map $s^i$ as 
\begin{itemize}
    \item the result of jiggling on $|K^i|$,
    \item $s^{(i-1)}$ on $|\cl(K \setminus \str(K^i))|$, and 
    \item the result of interpolation on each simplex in $\ring(K^i)$.
\end{itemize}
Note that the second item achieves in particular that $s^{(i)}$ equals $s$ over $|K' \cup K''|$. 

To ensure that the map remains a solution over $|\Op(K')|$, we need to make sure that, all together, the map $s$ is perturbed less than $\delta$. This can be achieved by choosing: 1) $\ell$ large enough so that the linearization is a good enough approximation, and 2) the constant $\epsilon_0$ small enough, which controls the size of all slope-perturbations in the proof of \cref{th:jigglingEucl}. For the second point we additionally use \cref{lem:interpol} to bound the perturbations over simplices where we interpolate.
\end{proof}

\subsection{Jiggling over manifolds} \label{sec:jigglingMfd}

We now move to the general case of jiggling sections of arbitrary fiber bundles $E \rightarrow M$. We will work chart by chart, relative to previous charts, using \cref{cor:jigglingEuclRel}. We note that if, in the statement below, $M$ is compact, the subdivision $T'$ is a crystalline subdivision of $T$, but not if $M$ is non-compact.

\jigglingMfd

\begin{proof}
Suppose first that $M$ is compact and $Q$ is empty. Since the set $B(s(M),\epsilon) \subset E$ is compact, it can be covered by finitely many fibered charts $\{\phi_i : U_i \times W_i \to \R^m \times \R^n \}_{i \in [I]}$, where $U_i$ is an open subset of $M$ and $W_i$ an open subset of the fiber. We can assume that $\ell$ is large enough such that we can cover $K_\ell$ by finitely many subcomplexes $\{K^{(i)}\}_{i \in [I]}$ (not necessarily nice) with $T(|\strN{2} K^{(i)}|) \subset U_i$.

The argument now amounts to doing induction on $i$. Starting with $s = s^{(-1)}$, we produce a sequence of sections $s^{(i)}$ via subsequent applications of \cref{cor:jigglingEuclRel}. We note that linear extensions and slope-perturbations work equally well for sections as they do for maps. Concretely, in the $i$th step, using \cref{cor:jigglingEuclRel}, we jiggle the section $\phi_i \circ s^{(i-1)} : T(|\strN{2}(K^{(i)})|) \to \trivbr$ relative to $T(|\cup_{j=0}^{i-1} K^{(j)}|)$ and $T(|\ring(\str(K^{(i)}))|)$. We extend the resulting section $s^{(i)}|_{T(|\strN{2}(K^{(i)})|)}$ to $M$, by requiring that $s^{(i)}$ equals $s^{(i-1)}$ over $M \setminus T(|\strN{2}(K^{(i)})|)$. Whenever subdivision is needed for the jiggling, we do not just subdivide $K^{(i)}$, but the entirety of $K$. In finitely many steps the proof is complete.

\pfstep{The relative case} In the relative case where $Q \neq 0$ we apply jiggling relative to $Q$ throughout the proof.

\pfstep{The non-compact case} 
If $M$ is non-compact, we take an exhaustion $\{C_k\}_{k\in \N}$ by compacts.
Inductively on $k$, we jiggle the section $s$ over the annulus $C_{k} \setminus C_{k-1}$ without changing the triangulation of $C_{k-2}$. In what follows we provide the details.

We set again $s^{(-1)} = s$ and $K^{(-1)}=K$ and produce a sequence of sections $s^{(k)}$ that are each a piecewise solution of $\SR$ on a neighborhood of $C_i$ with respect to the triangulation $T: |K^{(k)}| \to M$. We can assume, by possibly removing some compacts from the exhaustion, that $T(|\strN{3}(K \cap C_k,K)|) \subset C_{k+1}$.

Assume now that $s^{(k-1)}$ and $K^{(k-1)}$ have been constructed, then the strategy is illustrated in \cref{fig:noncompact}. We define the subcomplex $L^{(k)}$ of $K^{(k-1)}$ such that $T(|L^{(k)}|) = T(|K^{(k-1)}|) \cap (C_k\setminus C_{k-1})$, and whenever we take its star or ring in the following we do so as subcomplex of $K^{(k-1)}$. Using the compact case, we then jiggle $s^{(k-1)}$ over $T(|\strN{2}(L^{(k)})|)$ relative to $T(|\ring(\str(L^{(k)}))|)$. We note that we jiggle $s^{(k-1)}$ over a manifold that generally will have a boundary, but that this does not pose any issues.

Jiggling produces a section $s^{(k)}$ over $T(|\strN{2}(L^{(k)})|)$ that is piecewise smooth with respect to $T:|(\strN{2}(L^{(k)}))|\to M$ for some $\ell \in \N$. In particular $s^{(k)}$ is now a solution over a neighborhood of $C_k\setminus C_{k-1}$. We extend $s^{(k)}$ to $M$ by requiring it equals $s^{(k-1)}$ on $M \setminus T(|\strN{2}(L^{(k)})|)$. We can assume that the jiggling was small enough so that $s^{(k)}$ restricted to $T(|\strN{2}(L^{(k)})|) \cap C_{k-1}$ remains a solution. Then $s^{(k)}$ is a solution over a neighborhood of $C_k$.

We extend the $\ell$th crystalline subdivision of $\strN{2}(L^{(k)})$ to a subdivision $K^{(k)}$ of $K$ defined as
\begin{itemize}
    \item the $\ell$th crystalline subdivision of $\strN{2}(L^{(k)})$,
    \item the $0$th crystalline subdivision of $K^{(k-1)} \setminus (\strN{3}(L^{(k)} ))$, and
    \item the barycentric cone off in each simplex in $\ring( \strN{2}( L^{(k)} ) )$.
\end{itemize}
The section $s^{(k)}$ is then indeed piecewise smooth with respect to $T: K^{(k)}\to M$.

Since we assumed that $T(|\strN{3}(K \cap C_k,K)|) \subset C_{k+1}$, we see that $K^{(k)}$ and $K^{(k-1)}$ agree over $C_{k-2}$. In the limit we hence obtain a well-defined triangulation.
\end{proof}

\begin{figure}[h]
    \centering
    \includegraphics[width=0.8\linewidth]{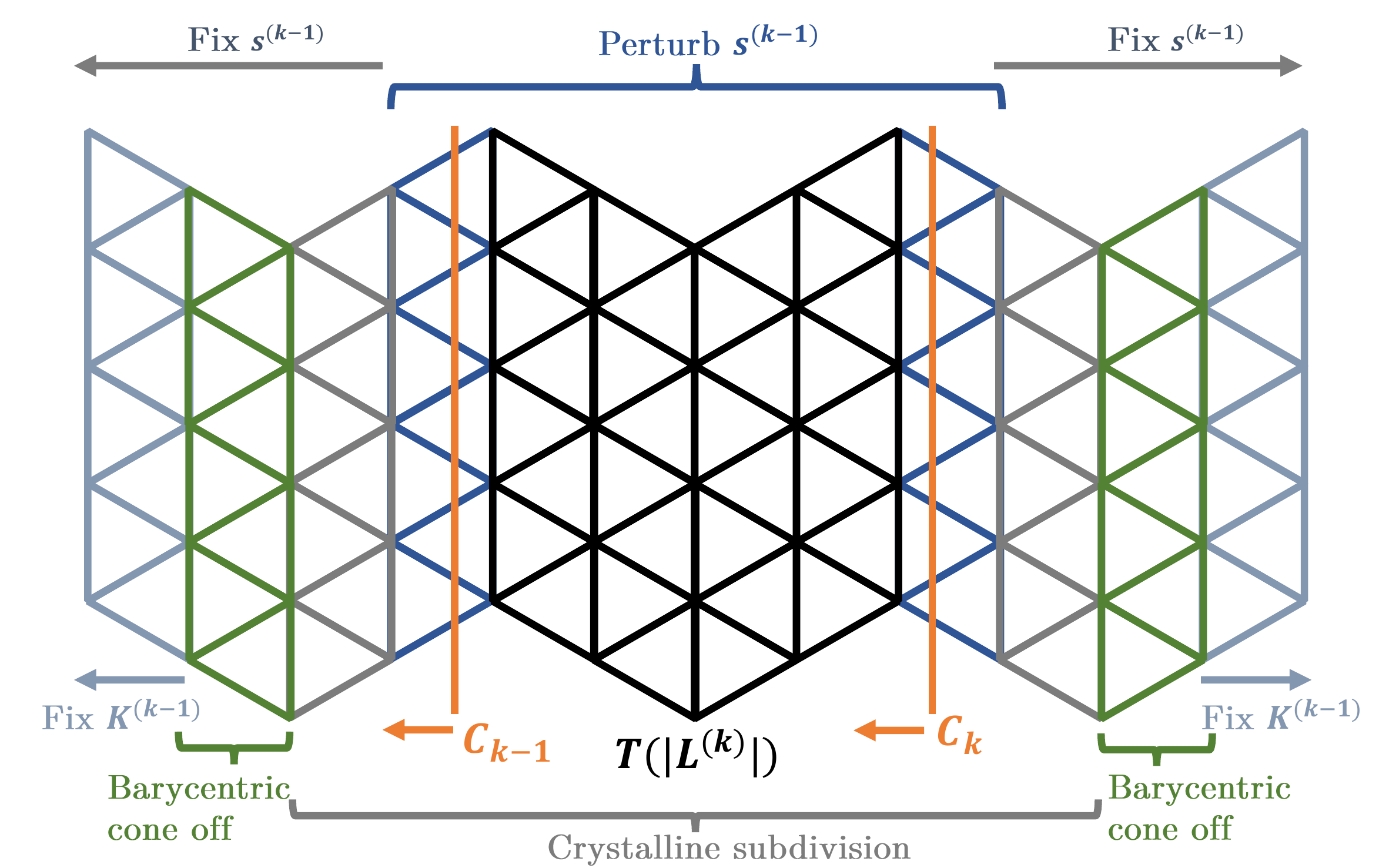}
    \caption{A sketch of the proof of \cref{th:jigglingMfd} in the non-compact case when about to construct $s^{(k)}$. The compacts $C_{k-1}$ and $C_k$ lie left of the respective orange lines. We have drawn $K^{(k-1)}$ and indicated the subcomplex $L^{(k)}$ and its stars. On top we indicate how to obtain $s^{(k)}$ from $s^{(k-1)}$ and on the bottom how to obtain $K^{(k)}$ from $K^{(k-1)}$.}
    \label{fig:noncompact}
\end{figure}

Alternatively, in the proof of \cref{th:jigglingMfd} for the non-compact case, we could have used \cref{prop:crystallineConeOff}. This makes sure that that the resulting subdivision of $K$ is a generalized crystalline subdivision, which proves useful in \cref{sec:JigglingGenPos}.
\begin{corollary} \label{cor:jigglingMfd} 
Consider the assumptions from \cref{th:jigglingMfd}. Then, there exists an $\epsilon$-jiggling $(s',T')$ such that
\begin{itemize}
    \item $s'$ is a piecewise solution of $\SR$ with respect to $T'$, 
    \item $s'|_ {Q} = s|_{Q}$, and
    \item the subdivision of $K$ underlying $T'$ is a generalized crystalline subdivision of $K$.
\end{itemize}
\end{corollary}
\begin{proof}
    The section $s'$ can be constructed as in the non-compact case of \cref{th:jigglingMfd}. For the subdivision $K'$ we have to make some changes to the proof, which we now indicate:

    Recall that $\{C_k\}_{k\in K}$ is an exhaustion by compacts of $M$. We define $\delta_k$ such that the closed $\delta_k$-neighborhood $\overline{B(C_{k},\delta)}$ of $C_{k}$ is contained in $|\str(K \cap C_{k},K)|$. 
    
    We now start the induction and assume that $s^{(k-1)}$ and $K^{(k-1)}$ have been constructed. We recall that the section $s^{(k)}$ is constructed such that it is piecewise smooth with respect to $T:|(\strN{2}(L^{(k)}))|\to M$, where $L^{(k)} = K^{(k)} \cap (C_k\setminus C_{k-1})$ is a subcomplex of $K^{(k)}$. We assume that $\ell$ is large enough that we can cone off $K^{(k-1)}_\ell \cap C_{k} $ within a $\delta_k$-neighborhood in $M$ to any finer crystalline subdivision of $K$ as in \cref{prop:crystallineConeOff}. If we do so for every $k$ (and in particular have done so for $k-2$), we can then define the subdivision $K^{(k)}$ of $K$ as the result of \cref{prop:crystallineConeOff} such that $K^{(k)}$ equals 
    \begin{itemize}
    \item the $0$th crystalline subdivision of $K^{(k-1)} \cap C_{k-2}$, and
    \item the $\ell$th crystalline subdivision of $K^{(k-1)}$ in $M \setminus B(C_{k-2},\delta_{k-2})$.
    \end{itemize}
    We see that $K^{(k)}$ and $K^{(k-1)}$ agree over $C_{k-2}$. In the limit we hence obtain a well-defined triangulation. The underlying crystalline subdivision is indeed a generalized crystalline subdivision as each simplex is only subdivided a finite number of times using crystalline subdivision, except when it lies in a $B(C_k,\delta) \setminus C_k$, then also the barycentric cone off is applied once, before it is no longer changed.
    
    We observe that the section $s^{(k)}$ is then indeed piecewise smooth with respect to $T: K^{(k)}\to M$.
\end{proof}

    \section{Jiggling as an h-principle} \label{sec:whe}

In this section we discuss how jiggling can be interpreted as an h-principle without homotopical assumptions for piecewise solutions. As we will see in \cref{sec:whesSet} (and as foreshadowed in \cref{sec:simplSets}), simplicial sets provide the natural setting for such an statement.

Even though the setting of topological spaces is less natural for our purposes (see \cref{sec:wheNotSpaces} for an informal discussion), we do want to make the following observation: at this point we have already proven the $\pi_0$-surjectivity of the map $\SolPS \hookrightarrow \GammaPS(E)$ between \emph{topological} spaces. Here we denote by $\GammaPS(E)$ the space of piecewise smooth sections of $E$ and by $\SolPS$ the subspace of piecewise smooth solutions of $\SR$.

\begin{corollary}
Let $\SR \subset J^1(E)$ be an open, fiberwise dense, differential relation. Then, every section $s: M \to E$ is homotopic, through piecewise smooth sections, to a piecewise smooth solution $s': M \to E$. 
\end{corollary}
\begin{proof}
We can $\epsilon$-jiggle $s$ using \cref{th:jigglingMfd} to end up with a piecewise smooth solution $s': M \to E$ with respect to a triangulation $T'$ of $M$. If we take $\epsilon$ small enough, the sections $s$ and $s'$ are homotopic as piecewise smooth maps by linearly interpolating (with respect to some fibered charts chosen a priori) over $T'$.
\end{proof}

\subsection{A weak homotopy equivalence between topological spaces} \label{sec:wheNotSpaces}

We now make a handful of remarks about the map $\iota: \SolPS \hookrightarrow \GammaPS(E)$. We do this to motivate the simplicial approach and pose a couple of open questions. This subsection contains no results and can otherwise be skipped.

Suppose we are interested in proving $\pi_k$-injectivity of $\iota$. To do so we start off with a family of sections $s: (D^n,S^{n-1}) \to (\GammaPS(E),\SolPS)$ representing a relative homotopy class. What we would like to do is consider the induced section $\tilde{s} : M \times D^n \to E$ and jiggle it. For each $t \in S^{n-1}$, the sections $\tilde{s}_t: M \times\{t\} \to E$ are piecewise smooth, but with respect to different triangulations $T_t$ of $M \times \{t\}$. However, without further assumptions, these triangulations need not be compatible in any way. In particular, even though the sections $\tilde{s}_t$ are piecewise smooth, the section $\tilde{s}$ need not be. It follows that we cannot apply jiggling.

Conversely, consider a section $s:M \times D^n \to E$ that is piecewise smooth with respect to a triangulation $T$ of $M \times D^n$. We can then wonder whether the sections $s_t: M \times \{t\} \to E$ are piecewise smooth. However, the restriction of $T$ to a fiber $M \times \{t\}$ need not be a triangulation. As such, the sections $s_t$ may not be piecewise smooth.

These remarks suggest the following approach, which we leave for the reader to explore. Suppose that we consider a triangulation $T$ of $M \times D^n$ that is \emph{in general position} with respect to the fibers $M \times\{t\}$ (recall \cref{def:genposTriang}). This may be enough to deduce that $T$ induces a triangulation on each fiber. Moreover, perhaps this can be used to show that every section that is piecewise with respect to $T$ defines piecewise smooth sections fiberwise. This ingredient may be enough to address the connectivity of the map $\iota$.

Instead of pursuing this, we explain in the next section how a weak homotopy equivalence of simplicial sets follows almost immediately from our earlier results.

\subsection{A weak homotopy equivalence of simplicial sets} \label{sec:whesSet}

We now turn our attention to the inclusion $\sSolPS \hookrightarrow \Sing(\Gamma^0(E))$ of the simplicial set of piecewise solutions into the simplicial set of continuous sections $E$.

We denote by $\Gamma^0(E)$ the space of continuous sections of $E$, endowed with the $C^0$-topology. We recall that $\Sing(\Gamma^0(E))$ is defined by $(\Sing(\Gamma^0(E)) )_n =  \{\Delta^n \to \Gamma^0(E)\}$, where the latter set is in bijection with the set $\{s: M \times \Delta^n \to q_n^*(E) \mid s \text{ continuous section}\}$ and $q_n: M \times \Delta^n \to M$ denotes the projection onto the first factor. The simplicial set $\sSolPS$ of piecewise solutions we define as the following.
\begin{definition} \label{def:sSolPS}
$\sSolPS \in \sSet$ is the simplicial set with  $n$-simplices 
	\begin{equation*}
		(\sSolPS)_{n} = \left\{ s: M \times \Delta^n \to q_n^*(E) \mid s \text{ is a piecewise smooth solution of } q_n^*(\SR) \right\}.
	\end{equation*}
The face map $\face{i}:(\sSolPS)_n \to (\sSolPS)_{n-1}$ is given by restricting the solution to $M \times F_i $ where $F_i$ is the face of $\Delta^n$ excluding the $i$th vertex, and by identifying this face $F_i$ with $\Delta^{n-1}$. The degeneracy map $\degen{i}:(\sSolPS)_n \to (\sSolPS)_{n+1}$ is induced by the pullback along the retraction $(\id, r_i) : M \times \Delta^{n+1} \to M \times \Delta^n$ where $r_i$ is the affine map collapsing the edge $\langle i, i+1 \rangle$.
\end{definition}

We make two remarks regarding the above definition. First, from the perspective of jiggling, it does not matter that we now consider solutions of the pullback relation $q_n^*(\SR)$. If $\SR \subset J^1(E)$ is an open and fiberwise dense relation, the relation $q_n^*(\SR) \subset J^1(q_n^*(E))$ is as well. Secondly, we observe that $\sSolPS$ is a Kan complex.
\begin{lemma} \label{lem:ssSolisKan}
    The simplicial set $\sSolPS$ is a Kan complex.
\end{lemma}
\begin{proof}
    Let a morphism $\Lambda[n,k] \to \sSolPS$ be given. This means we are given a piecewise smooth solution $s$ of $\SR$ on $ M  \times \Lambda^n_k$ which is piecewise smooth with respect to a triangulation $T$. 
    Extend $s$ to a piecewise smooth section $\tilde{s}$ of $E$ over $M \times \Delta ^n$ such that it is piecewise smooth with respect to a triangulation $\tilde{T}$ that extends $T$. We assume that $\tilde{s}$ is a solution on $\Op(M \times \Lambda^n_k)$. Now we jiggle $(\tilde{s},\tilde{T})$ relative to $M \times \Lambda^n_k$ using \cref{th:jigglingMfd}, which gives the requested extensions of $s$.
\end{proof}

To establish a weak homotopy equivalence between $\sSolPS$ and $\Sing(\Gamma^0(E))$, we just need to unpack what this entails in terms of sections of the bundle $E$. The statement will then follow almost immediately from jiggling as in \cref{th:jigglingMfd}. The reason is that the parametric version of jiggling, in this simplicial sense, is just a specific case of the non-parametric case. In the proof we will use the fact that any triangulation of $M \times \partial \Delta^n$ can be extended to a triangulation of $M \times \Delta^n$; this follows from the contractibility of the space of triangulations~\cite{lurie2009topics}. 

\wheSSets
\begin{proof}
According to \cref{def:wheTriangleHomotopy}, we have to show that, given a diagram of the form 
	\[ \begin{tikzcd}
		\partial \Delta[n] \arrow{r} \arrow{d} & \sSolPS \arrow{d} \\%
		\Delta[n] \arrow{r} \arrow[dashrightarrow]{ur} & \Sing(\Gamma^0(E)),
	\end{tikzcd}
	\]
	there exists a lift $\Delta[n] \to \sSolPS$ such that the upper triangle commutes and the lower triangle commutes up to homotopy.
	
	Hence we assume that we are given a continuous section $F: M \times \Delta^n \to q_n^*(E)$ such that $F$ restricted to $M \times \partial \Delta^n$ is a piecewise smooth solution of $q_n^*(\SR)$. Then we need to find a section $G: M \times \Delta^n \to q_n^*(E)$ such that
	\begin{itemize}
		\item $G$ is a piecewise smooth solution of $q_n^*(\SR)$, and
		\item $G$ is homotopic to $F$ via continuous sections relative to $M \times \partial \Delta^n$.
	\end{itemize}
	
	The construction of $G$ consists of two steps. Firstly, we approximate $F$ by a section $F'$ that is piecewise smooth such that $F'$ is a solution on $\Op(M \times \partial \Delta^n)$ and such that $F'$ agrees with $F$ on $M \times \partial \Delta^n$. Secondly, we apply jiggling of \cref{th:jigglingMfd} to $F'$ relative to $M \times \partial \Delta^n$ to end up with $G$. If $F'$ and $F$, and in turn $G$ and $F'$ are sufficiently close, they are all homotopic via continuous sections.
\end{proof}

\begin{remark}
    We want to remark that it is also an option to define $\sSolPS$ as consisting of pairs $(s,T)$ where additionally $s$ is piecewise smooth with respect to the triangulation $T$ of $M \times \Delta^n$. The face map can be defined as the restriction. One way of defining the degeneracy maps is noting that the $n$-simplices and the face maps form a semi-simplicial set that is Kan, which can hence be endowed with degeneracy maps (and degenerate simplices) which make it into a simplicial set $\sSolPST$~\cite{rourke1971delta}. More concretely, the simplicial set can be described as the left Kan extension of the semi-simplicial set. With minor adaptations the proof of \cref{th:wheSSets} follows through, proving that the inclusion $\sSolPST \hookrightarrow \Sing(\Gamma^0(E))$ is a weak homotopy equivalence. In particular this recovers that the fibers of the forgetful functor $\sSolPST \to \sSolPS$ are contractible.
\end{remark}

	\section{Applications} \label{sec:app}

In this section we discuss two applications of jiggling. We start with maps that are transverse to a distribution in \cref{sec:appTrans}, which has applications to both immersions and submersions. After that we discuss contact forms in \cref{sec:appContact}. Finally, in \cref{sec:GenPosDef} we discuss how jiggling can be applied to triangulation to become in general position to a distribution.

\subsection{Transversality}\label{sec:appTrans}

As a first application we show that using jiggling every map can be made to be piecewise transverse to a distribution. For compact (subsets of) manifolds this follows from Thurston's jiggling lemma, but here we also deal with the non-compact case.
\begin{corollary}\label{cor:appTransMap}
	Let $\epsilon: M \to \R_+$ be given, let $f: M \to N$ be a piecewise smooth map and $\xi$ a distribution of constant rank on $N$. Then there exists a map $g : M \to N$ that is transverse to $\xi$ and satisfies $\dist1 (f,g)<\epsilon$. In particular we can choose $f$ and $g$ to be homotopic as piecewise smooth maps.
\end{corollary}
\begin{proof}
    A map $M \to N$ induces a section $M \to M \times N$, which we can jiggle using \cref{th:jigglingMfd} to be transverse to $\xi$, since this is an open and fiberwise dense relation on $J^1(M \times N)$.
\end{proof}

In particular this result can be applied to maps of maximal rank by the observation that a map $M \to N$ is of maximal rank if and only if the induced section $M \to M \times N$ is transverse to the horizontal distribution.
\begin{corollary}\label{cor:appMapMaxRank}
	Let $\epsilon: M \to \R_+$ and let $f:M \to N$ be a smooth map. Then there exists a piecewise smooth map $g: M \to N$ such that $g$ is a piecewise embedding of maximal rank and $\dist1 ( g , f) < \epsilon$. In particular we can choose $f$ and $g$ to be homotopic as piecewise smooth maps.
\end{corollary}
\begin{proof}
	We define the distribution $\xi_H: M \times N \to T(M \times N)$ as the horizontal distribution $\xi_H(x,y) = T_y N$ and we denote by $\tilde{f} : M \to M \times N$ the section induced by $f$. Then we observe that that $f$ has maximal rank if and only if $\tilde{f}$ is transverse to $\xi_H$. Hence we apply \cref{th:jigglingMfd} to $f$ which results in a piecewise smooth map $\tilde{g}: M \to M \times N$ which is transverse to $\xi_H$. We can assume $\tilde{g}$ to be a section if we take $g$ and $f$ to be close enough. The requested map $g$ is then the one induced by $\tilde{g}$.
\end{proof}

In the above examples we also obtain a weak homotopy equivalence using \cref{th:wheSSets}. Here we denote the relation that a map $M \to N$ is transverse to a given distribution by $\SRtransv$, the relation that the map is an immersion by $\SRimm$ and that it is a submersion by $\SRsubm$.

\wheMaps

\subsection{Contact forms} \label{sec:appContact}

Throughout this section we assume that $M$ is a $(2n+1)$-dimensional manifold. We recall that a contact form\footnote{Results for contact structures analogous to the ones we state here for contact forms also hold by taking the projectivization.} is a $1$-form $\alpha \in \Omega^1(M)$ such that $\alpha \wedge (d \alpha)^n \neq 0$. Equivalently, a contact form is a section $\alpha: M \to \Lambda^1 (T^*M)$ whose $1$-jet satisfies the contact relation $\SRcont \subset J^1(\Lambda^1 (T^*M) \setminus\{0\})$. Note that we restrict to $\Lambda^1 \left(T^*M \right)\setminus\{0\} $ since the contact relation is not fiberwise dense when considering also vanishing $1$-forms.

\begin{corollary} \label{cor:contactPi0}
	Let $\epsilon : M \to \R_+ $ be given and let $M$ be a manifold of odd dimension with a non-vanishing $1$-form $\alpha \in \Omega^1(M)$. Then there exists a piecewise smooth contact form $\tilde{\alpha} \in \Omega^1(M)$ such that $\dist1 (\alpha,\tilde{\alpha}) < \epsilon$. In particular, we can choose $\alpha$ and $\tilde{\alpha}$ to be homotopic as piecewise smooth $1$-forms.
\end{corollary}

We compare the above result to h-principles \emph{with} homotopical assumptions for contact forms, that is, h-principles which compare the space of solutions (here: contact forms) to the space of formal solutions (here: formal contact forms). Recall that the space of formal contact forms consists of pairs $(\alpha, \beta) \in \Omega^1(M) \times \Omega^2(M)$ such that $\alpha \wedge \beta^n \neq 0$. 

Gromov~\cite{gromov1969stable} proved that if $M$ is an open manifold the complete h-principle holds. It tells us in particular, if we restrict ourselves to the case of $\pi_0$-surjectivity, that any formal contact form $(\alpha,\beta)$ is homotopic to a contact form $(\tilde{\alpha},d\tilde{\alpha})$ via formal contact forms in an open manifold. If we compare with \cref{cor:contactPi0}, we see that jiggling only requires a non-degenerate $1$-form $\alpha$ instead of also a $2$-form $\beta$ such that together they form a formal contact form. However, the output of jiggling is also just a piecewise smooth contact form instead of a smooth one.

For closed manifolds $\pi_0$-surjectivity was shown to hold by Martinet~\cite{Mart} and Lutz~\cite{Lutz}, but $\pi_0$-injectivity was proven to be false by Bennequin~\cite{Be}. Hence for closed manifolds the full h-principle does not hold, unless we restrict ourselves to \emph{overtwisted} contact structures as shown by Eliashberg~\cite{El89}.

Hence we see that the main strength of jiggling is that it also applies to closed manifolds. The full h-principle we prove, albeit without homotopical assumptions and for simplicial sets, holds true both for open and closed manifolds. 

\appContactWhe

\subsection{General position via jiggling} \label{sec:JigglingGenPos}

Our method of jiggling can also be used to construct triangulations in general position with respect to a distribution, which was Thurston's original statement. However, this claim does not follow directly from the statement of \cref{th:jigglingMfd}. The difficulty is that, even though the differential relation encoding general position lives in the first jet space, it depends on the triangulation.

\subsubsection{General position} \label{sec:GenPosDef}

We first recall:
\begin{definition}
	A smooth map $f: M \to N$ is \textbf{transverse to a distribution} $\xi$ of rank $r$ on $N$ if for all $x \in M$ the space $(df)_x(TM) + \xi_{f(x)}$ is maximal. That is, we require that 
	\[ \dim \left((df)_x(TM) + \xi_{f(x)}\right) =  \min\{n, m + r\}, \]
	where $m$ and $n$ denote the dimensions of $M$ and $N$ respectively.

    If $f$ is piecewise smooth, we require that the above holds on every top-dimensional simplex.
\end{definition}

Observe that this does not imply that the map restricted to a lower dimensional simplex is also transverse to the distribution. 
Hence we consider a stronger notion than transversality, which is known as general position. This requires not just the top-dimensional simplices to be transverse but also all lower dimensional ones. Consequently, it also requires that the distribution does not vary too much within a simplex. 

\begin{figure}[h]
    \centering
    \begin{subfigure}[t]{0.45\linewidth}
        \centering
        \includegraphics[width=\linewidth,page=2]{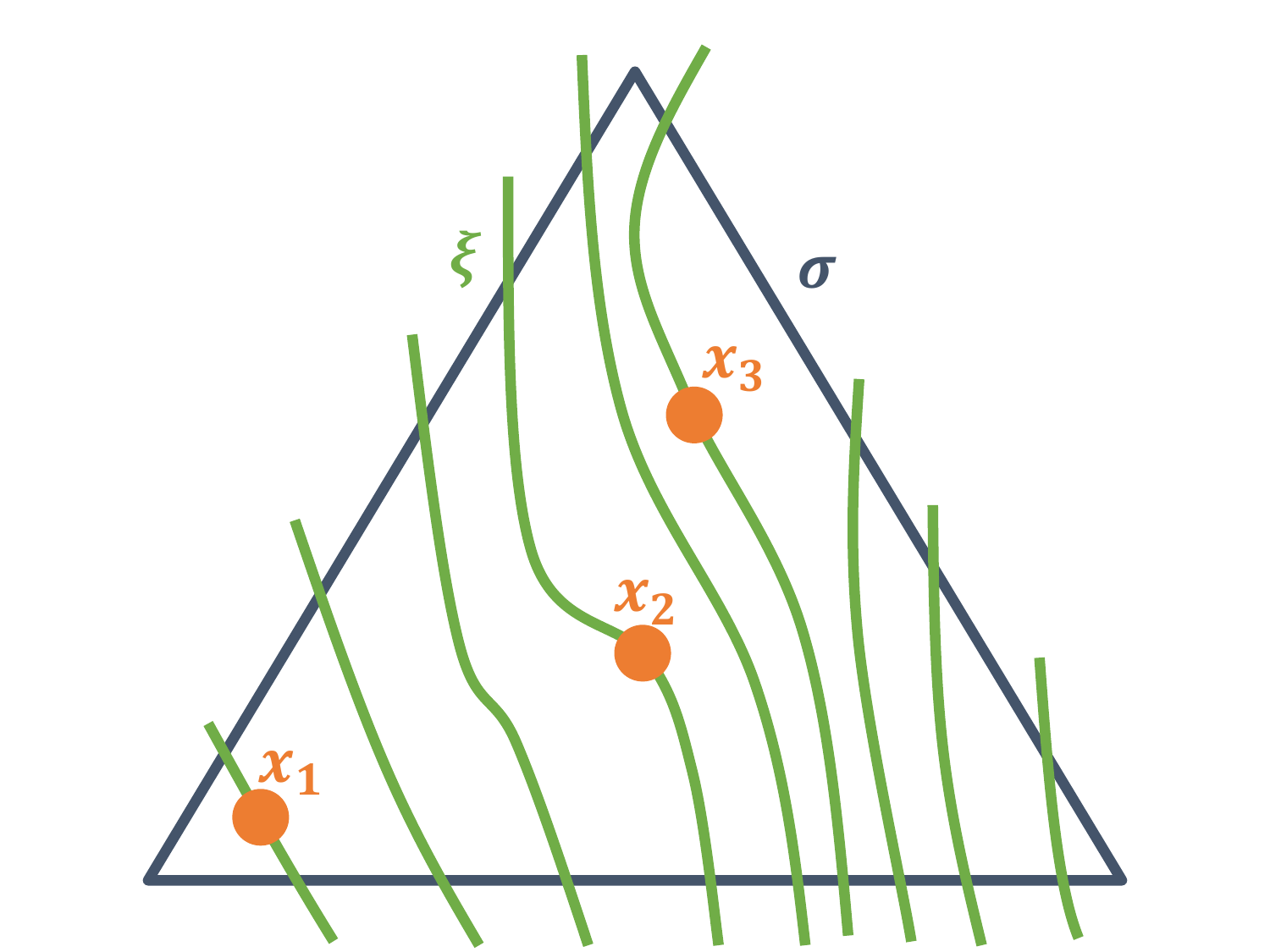}
        \caption{Here the simplex is not in general position, because the edge $\sigma$ is not transverse to $\folconst{ \xi_{x_i}}$.}
        \label{fig:exNotInGenPos}
    \end{subfigure}
    \hspace{0.05\linewidth}
    \begin{subfigure}[t]{0.45\linewidth}
        \centering
        \includegraphics[width=\linewidth,page=3]{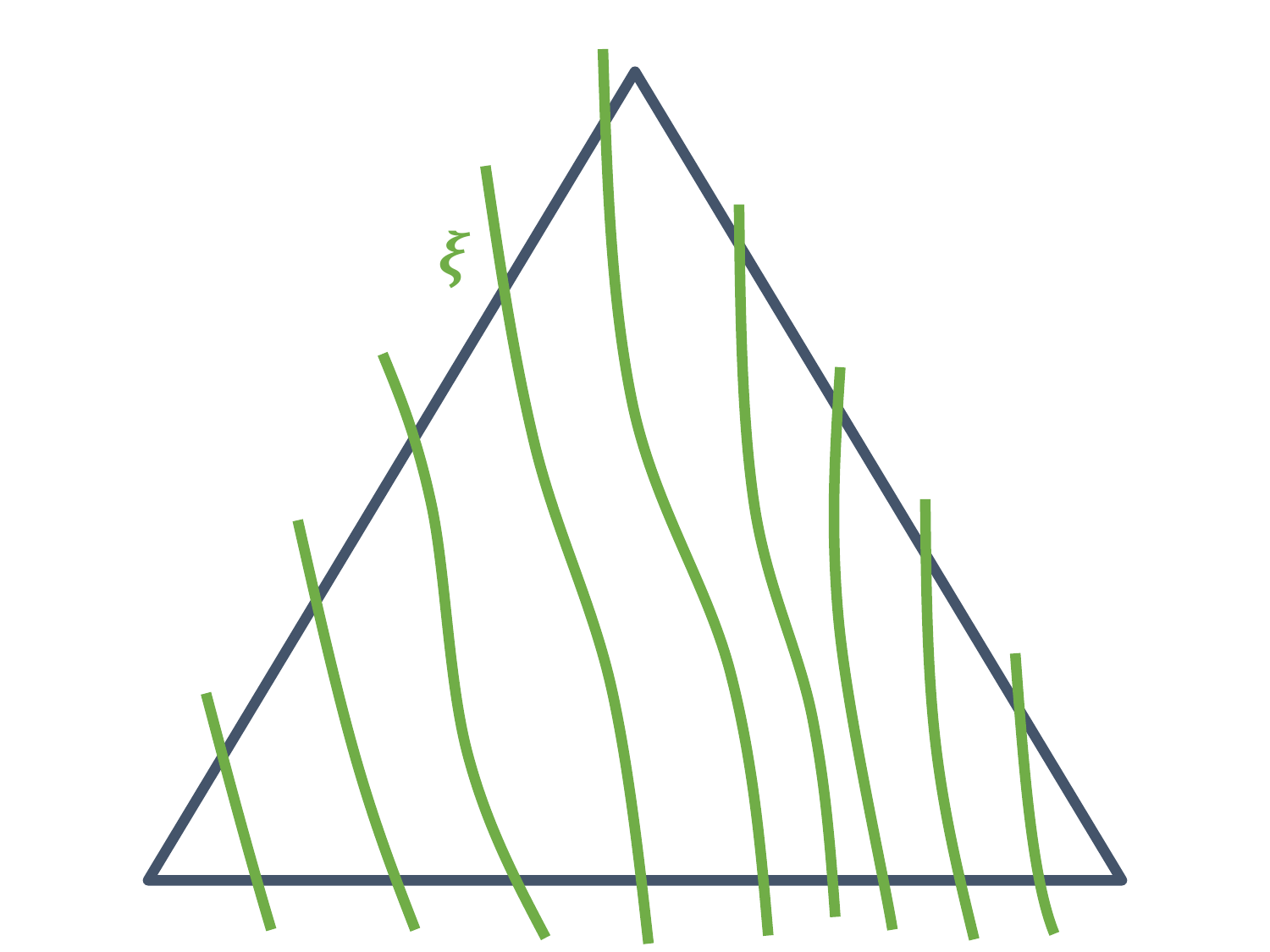}
        \caption{Here the simplex is in general position with respect to $\xi$, because the three edges are transverse to $\folconst{\xi_x}$ for each $x\in \Delta^2$.}
        \label{fig:exInGenPos}
    \end{subfigure}
	\centering
	\caption{We give two examples of a distribution $\xi$ on $\Delta^2$ such that $\Delta^2$ is (not) in general position with respect to $\xi$.} \label{fig:genpos}
\end{figure}

In the following definition we identify a top-dimensional linear simplex $\Delta$ in a simplicial complex $K$ with the standard simplex $\Delta^m \subset \R^m$. We also use the fact that any plane $V$ in a tangent space $T_x\R^m$ can be extended, using the Euclidean parallel transport, to a foliation over the whole of $\R^m$. We will denote this foliation by $\folconst{V}$. We illustrate the definition in \cref{fig:genpos}.
\begin{definition} \label{def:genpos}
Let $f: M \to N$ be a map that is piecewise smooth with respect to a triangulation $T: |K| \rightarrow M$. Let $\xi$ be a distribution of constant rank on $N$. Then $f$ is \textbf{in general position} with respect to $\xi$ if:
\begin{itemize}
    \item $f$ is transverse to $\xi$, and
    \item for each top-dimensional simplex $\Delta \in K$ and each face $\Delta'$, we have that $\Delta' \trans \folconst{(f \circ T)^*\xi_x}$ for all $x \in (f \circ T)(\Delta)$.
\end{itemize}   
\end{definition}

As a special case of the above definition we can also define what it means for a triangulation $T$ to be in general position.
\begin{definition}\label{def:genposTriang}
Let $T: |K| \to M$ be a triangulation and let $\xi$ be a distribution of constant rank on $M$. Then $T$ is in general position with respect to $\xi$ if $\id: M \to M$ is in general position with respect to $\xi$ and $T$.
\end{definition}
The general position condition constrains maps $M \to N$ in a manner that depends on the chosen triangulation $T$. In particular, a map in general position may stop being so upon subdivision (see \cref{fig:SubdivGenPos}). Jiggling on the other hand relies on being able to take sufficiently fine crystalline subdivisions $T_\ell$ of $T$. Hence this is something that we must address.

\begin{figure}[h]
	\includegraphics[width=0.5\textwidth,page=6]{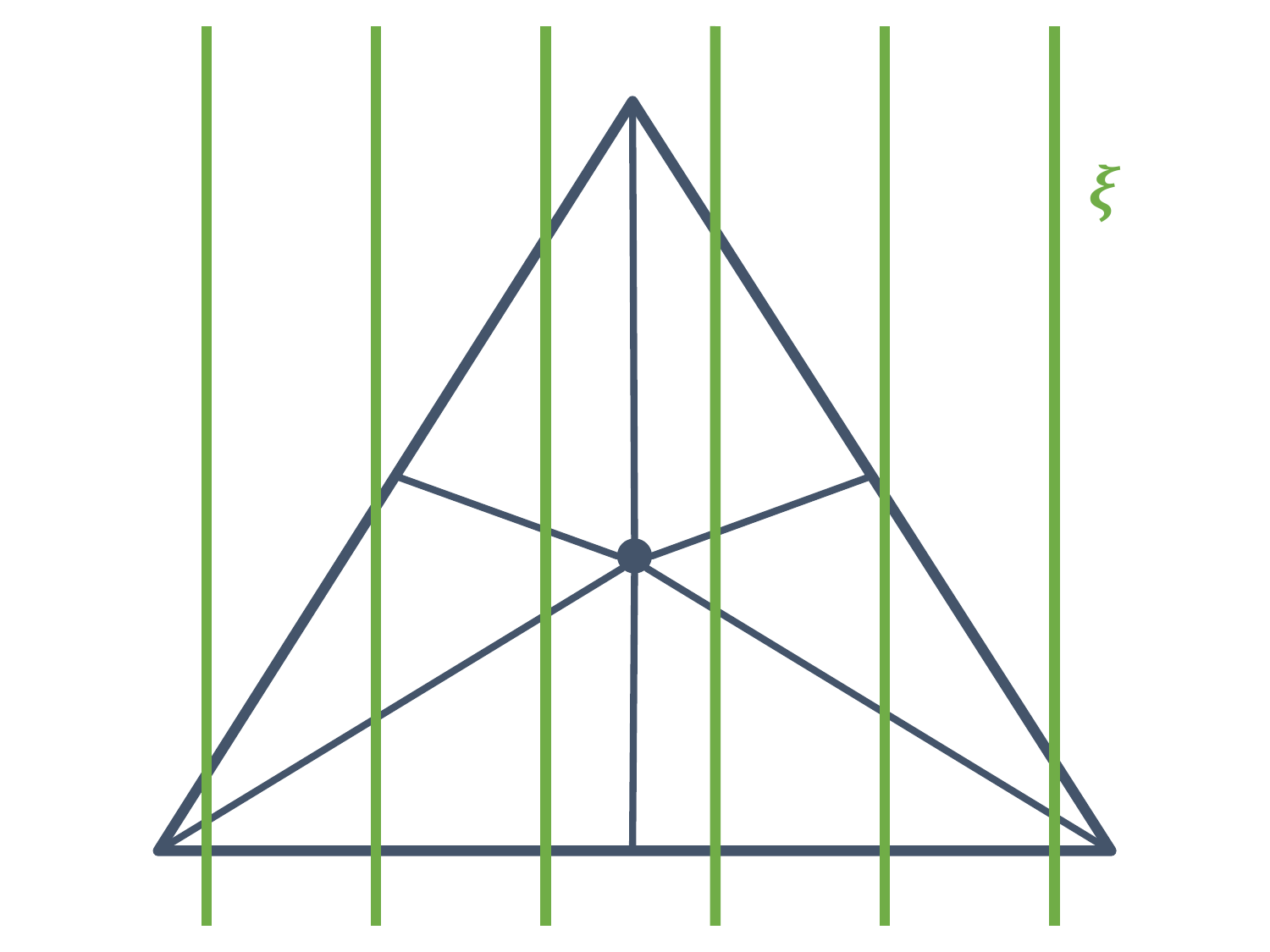}
	\centering
	\caption{Example of a triangulation in general position with respect to the vertical distribution $\xi$ whose subdivision is not.} \label{fig:SubdivGenPos}
\end{figure}

\subsubsection{Jiggling} \label{sec:JigglingGenPosTriang}

We now fit general position into our jiggling framework. We are given a map $f: M \to N$ that is piecewise smooth with respect to a triangulation $T: K \rightarrow M$ and we fix a distribution $\xi$ of constant rank on $N$. In order to jiggle, we introduce the notation $\SC_{\Delta}$ for the collection of model simplices for a single top-dimensional simplex $\Delta \in K$ under generalized crystalline subdivisions of $K$. That is, each simplex in a generalized crystalline subdivision of $K$ that is a subset of $\Delta$ equals a model simplex in $\SC_\Delta$ up to translation and scaling. The scaling factor is at this moment of no importance. 

The collection $\SC_\Delta$ is finite for each $\Delta$, as we can construct it as in the proof of \cref{prop:crystallineConeOff}: by \cref{lem:csubdivExact} there exists a finite set of model simplices $\SC'_\Delta$ for crystalline subdivisions of $\Delta$. Now given a model simplex $\Delta'\in \SC'_\Delta$, we apply crystalline subdivision at most once to each of its faces. We then use the barycentric cone off to extend the subdivision of the boundary of $\Delta'$ to $\Delta'$ itself. We add each of the resulting simplices to $\SC'_\Delta$, to define the finite collection $\SC_\Delta$. This motivates the following definition, of a notion stronger than general position.

\begin{definition}
Let $M$ and $N$ be smooth manifolds and fix a triangulation $T: K \to M$. We define the \textbf{very general position} relation $\SRverygenpos \subset J^1(M,N)$ as follows. A jet $\sigma \in J^1(M,N)$ with front projection $ (x,y) = \pi_f(\sigma) \in M \times N$ belongs to $\SRverygenpos$ if
\begin{itemize}
\item for every $\Delta \in K$ with $x \in T(\Delta)$, and
\item every simplex $\Delta'$ in $\SC_\Delta$,
\end{itemize}
it holds that $\Delta'$ is transverse to $\folconst{(\sigma \circ T)^*\xi_y}$.
\end{definition}
Here we use the fact that the pullback of a plane field, at a point, only depends on the plane field and the first jet of the map at said point. We now observe the following.

\begin{lemma} \label{lem:genposIsSFDL}
Fix $M$, $N$, $K$, and $\xi$ as above, and fix a generalized crystalline subdivision $K'$ of $K$. Assume that $\SC_\Delta$ is finite for each top dimensional simplex $\Delta \in K$. Then the following hold:
\begin{itemize}
	\item The relation $\SRverygenpos$ is open and fiberwise dense.
	\item If $s': M \rightarrow N$ is a piecewise smooth solution of $\SRverygenpos$ with respect to $K'$, then it is in general position with respect to all generalized crystalline subdivisions $K''$ of $K$ that also subdivide $K'$.
\end{itemize}
\end{lemma}
\begin{proof}
Openness and fiberwise density follow immediately from the finiteness of all $\SC_\Delta$ and the local finiteness of $K$. The second claim follows from the definition of $\SC_\Delta$ as the collection of model simplices for $\Delta \in K^{(\topd)}$ under generalized crystalline subdivisions of $K$. 
\end{proof}

The following is now a consequence of \cref{cor:jigglingMfd}. We point out that \cref{th:jigglingMfd} cannot be applied to obtain the following result, since the relation $\SRverygenpos$ depends on the model simplices for top-dimensional simplices in $K$, which change throughout the proof of \cref{th:jigglingMfd} due to the coning off of arbitrarily many subdivisions of the boundary.
\begin{corollary} \label{cor:SectionGenPos}
    Let $M$ and $N$ be manifolds, possibly open, where $M$ is  endowed with a smooth triangulation $T: K \to M$ and $N$ with a distribution $\xi$ of constant rank. Then, given
    \begin{itemize}
        \item a continuous function $\epsilon : M \to \R_+$,
        \item a map $s: M \rightarrow N$, and
        \item a subcomplex $Q \subset M$ of $T$ such that $s|_{\Op (Q)}$ is a solution of $\SRverygenpos$,
    \end{itemize}
    there exists an $\epsilon$-jiggling $(s',T')$ of $(s,T)$ such that
    \begin{itemize}
        \item $s'$ is in (very) general position with respect to $\xi$,
        \item $s'|_{Q} = s|_{Q}$, and 
        \item $T'$ is a generalized crystalline subdivision of $T$.
    \end{itemize}
\end{corollary}
\begin{proof}
    We have already observed that $\SC_\Delta$ is finite for each top-dimensional simplex $\Delta \in K$ and hence by the first item of \cref{lem:genposIsSFDL}, the relation $\SRverygenpos$ is open and fiberwise dense. This allows us to jiggle $(s,T)$ using \cref{cor:jigglingMfd}. If we check the proof of \cref{cor:jigglingMfd}, we see that it produces a sequence of sections $s^{(k)}$ that are piecewise smooth with respect to a triangulation $T:|K^{(k)}| \to M$, and are a piecewise smooth solution over $C_k$, where $\{C_k\}_{k\in K}$ is an exhaustion by compacts of $M$. Moreover, the functions $s^{(k)}$ and $s^{(k-1)}$ agree over $C_{k-2}$, and $K^{(k)}$ is a generalized subdivision of $K$ and a subdivision of $K^{(k-1)}$. Hence by the second item of \cref{lem:genposIsSFDL}, we see that the pair $(s',T')$ produced by \cref{cor:jigglingMfd} satisfies the current statement.
\end{proof}

It is also immediate that the same statement applies for general position with respect to a finite number of distributions. 

\subsubsection{Triangulations in general position}
The above result can then be applied to the identity map $\id : M \to M$, which is the setup for Thurston's jiggling.  The resulting jiggled map will be a piecewise immersion by construction. The fact that it is a homeomorphism follows from degree theory and is immediate if we take a small enough jiggling~\cite[Lecture 4]{lurie2009topics}. Thus:
\begin{corollary} \label{cor:jigglingTriangOpen}
Let $M$ be a manifold triangulated by $T: |K| \to M$ and endowed with a distribution $\xi \subset TM$. Then given
	\begin{itemize}
		\item a continuous function $\epsilon: M \to \R_+$, and
		\item a subcomplex $K'\subset K$ such that $T|_{\Op (|K'|)}$ is in very general position,
	\end{itemize} 
there exists an $\epsilon$-jiggling $T'$ of $T$ such that
\begin{itemize}
    \item $T'$ is in (very) general position with respect to $\xi$, and
    \item $T'|_{|K'|} = T|_{|K'|}$.
\end{itemize}
\end{corollary}

	\printbibliography
\end{document}